\documentclass[11pt, leqno]{article} 
\usepackage{amsthm}
\usepackage{amssymb,amsmath}
\usepackage{graphics}

\newcommand{\Div}{\mathop{\rm div}\nolimits}
\newtheorem{theorem}{Theorem}
\newtheorem{corollary}{Corollary}
\newtheorem{lemma}{Lemma}
\newtheorem{proposition}{Proposition}

\newtheorem{definition}{Definition}

\newtheorem{algorithm}{Algorithm}
\newtheorem{remark}{Remark}

\newcommand{\cT}{\mathcal{T}}
\newcommand{\R}{\mathbb{R}}
\newcommand{\Rn}{{\mathbb{R}^n}}
\newcommand{\N}{\mathbb{N}}

\newcommand{\om}{\Omega }
\newcommand{\pom}{{\partial \Omega} }
\newcommand{\omc}{ {\Omega^c} }
\newcommand{\gs}{{\Gamma_s}}
\newcommand{\gt}{{\Gamma_t}}
\newcommand{\Wgs}{{\widetilde{W}^{\frac{1}{2},2}(\gs)}}
\newcommand{\Whgs}{{\widetilde{W}_h^{\frac{1}{2},2}(\gs)}}
\newcommand{\Whg}{{W_h^{\frac{1}{2},2}(\pom)}}
\newcommand{\Wg}{{W^{\frac{1}{2},2}(\pom)}}
\newcommand{\Wmg}{{W^{-\frac{1}{2},2}(\pom)}}
\newcommand{\Whmg}{{W_h^{-\frac{1}{2},2}(\pom)}}
\newcommand{\Wp}{{W^{1,p}(\om)}}
\newcommand{\Whp}{{W_h^{1,p}(\om)}}

\begin{document}

\title{Adaptive FE--BE Coupling for Strongly Nonlinear Transmission Problems with Coulomb Friction}
\author{H.~Gimperlein, M.~Maischak,
  E.~Schrohe, E.~P.~Stephan}
\maketitle 
\begin{abstract}
\noindent We analyze an adaptive finite element/boundary element
procedure for scalar elastoplastic interface problems involving
friction, where a nonlinear uniformly monotone operator such as the
$p$--Laplacian is coupled to the linear Laplace equation on the
exterior domain. The problem is reduced to a boundary/domain
variational inequality, a discretized saddle point formulation of
which is then solved using the Uzawa algorithm and adaptive mesh
refinements based on a gradient recovery scheme. The Galerkin
approximations are shown to converge to the unique solution of the
variational problem in a suitable product of $L^p$-- and
$L^2$--Sobolev spaces.
\end{abstract}

\section{Introduction}\label{sec:intro}

Consider the following transmission problem on a bounded Lipschitz domain $\om \subset \Rn$:
\begin{align}
-\mathrm{div} \left( \varrho(|\nabla u_1|) \nabla u_1\right)= f \quad \text{in $\Omega$,}
- \Delta u_2 =&0 \quad \text{in $\Omega^c$,} \nonumber\\
\varrho(|\nabla u_1|) \partial_\nu u_1-\partial_\nu u_2= t_0 \quad \text{on $\partial \Omega$,}
u_1-u_2= &u_0\quad \text{on $\Gamma_t$,} \nonumber\\
-\varrho(|\nabla u_1|) \partial_\nu u_1 (u_0+u_2-u_1) +
g|(u_0+u_2-u_1)|=&0, \label{diff}\\
  \left| \varrho(|\nabla u_1|) \partial_\nu u_1\right| \leq g \quad \text{on $\Gamma_s$.}\nonumber\\
u_2(x)=\left\{\begin{array}{l@{}l}
      a+o(1)&,n=2 \\
     \mathcal{O}(|x|^{2-n})&,n>2
      \end{array}\right.\nonumber.&
\end{align}
Here $\varrho(t)$ denotes a function $\varrho(x, t)\in C(\overline{\Omega} \times (0,\infty))$ satisfying $$0 \leq  \varrho(t) \leq \varrho^* [t^\delta (1+t)^{1-\delta}]^{p-2},$$
$$|\varrho(t)t-\varrho(s)s| \leq \varrho^* [(t+s)^\delta (1+t+s)^{1-\delta}]^{p-2}|t-s|$$
and $$\varrho(t)t-\varrho(s)s \geq \varrho_* [(t+s)^\delta
(1+t+s)^{1-\delta}]^{p-2}(t-s)$$ for all $t \geq s >0$ uniformly in
$x \in \Omega$ ($\delta \in [0,1]$, $\varrho_*, \varrho^* >0$). The
interface $\pom = \overline{\gs \cup \gt}$ is divided into the
disjoint components $\gs$ and $\gt \neq \emptyset$, and the data
belong to the following spaces:
$$f\in L^{p'}(\Omega),\ u_0 \in W^{\frac{1}{2},2}(\pom),\ t_0 \in W^{-\frac{1}{2},2}(\pom), \ g \in L^\infty(\gs), \ a \in \R.$$
As usual, the normal derivatives are understood in terms of a Green's formula, and it is convenient to set $a=0$ for $n>2$. In two dimensions one further condition is required to enforce uniqueness:
\begin{equation}\label{uniq}
\int_\Omega f + \langle t_0, 1 \rangle=0.
\end{equation}
We are looking for weak solutions $(u_1,u_2) \in W^{1,p}(\om) \times W^{1,2}_{loc}(\omc)$ when $p \geq 2$.
A typical example is given by $\varrho(t)=[t^\delta (1+t)^{1-\delta}]^{p-2}$, $\delta \in [0,1]$, with the $p$--Laplacian corresponding to the maximally degenerate case $\delta =1$.

In this article we use layer potentials for the Laplace equation on $\omc$ to reduce the system to a uniquely solvable variational problem on $W^{1,p}(\om) \times W^{\frac{1}{2},2}_0(\gs)$. The main idea of our theoretical analysis is simple: Because the traces of $W^{1,p}(\om)$--functions are continuously embedded into $W^{\frac{1}{2},2}(\pom)$ for $p \geq 2$, the quadratic form $\langle S u, u\rangle$ associated to the Steklov--Poincar\'e operator is accessible to Hilbert space methods whenever it is defined. In this slightly weaker setting, Friedrichs' inequality (Prop.~\ref{poinfried}) allows to recover control over the $L^p$--norms in the interior, and as a consequence the full variational functional associated to the above equations is coercive in $W^{1,p}(\om)$.

In the numerical part we present a model problem, which shows
singularities resulting from the given boundary data, as well as from
the change of boundary conditions, leading to a suboptimal convergence
rate for uniform mesh refinements. We also present a Uzawa solver to
deal with the variational inequality.

With the help of a Korn inequality (Prop.~\ref{korn}), our method
easily carries over to transmission problems in nonlinear
elasticity, e.g.~Hencky materials in $\Omega$ coupled to the Lam\'e
equation in $\Omega^c$. A generalization to certain nonconvex energy
functionals will be discussed elsewhere \cite{micro}.

The outline of the article is as follows: Section~\ref{sec:pre}
recalls some properties of $L^p$-Sobolev spaces and introduce a
family of quasinorms adapted to the considered class of operators.
In the following section~\ref{formulation} we introduce the boundary
integral operators and derive our variational formulation.
Section~\ref{sec:uniq} is dedicated to the existence and uniqueness
of our model problem. The discretization of our problem is derived
in section~\ref{sec:disc}, as well as the a-priori error estimates.
In section~\ref{sec:post} our a-posteriori error estimator is
presented and its reliability proven. Finally, in
section~\ref{sec:num} we present the Uzawa-solver and two numerical
examples, clearly underlining our theoretical results.

\section{Preliminaries}\label{sec:pre}

Let $\om$ be an open subset of $\Rn$ with Lipschitz boundary $\pom$. Set $p'=\frac{p}{p-1}$ whenever $p \in (1,\infty)$.

\begin{definition}
The Sobolev spaces $W_{(0)}^{k,p}(\om)$, $k \in \N_0$, are the completion of $C_{(c)}^\infty(\om)$ with respect to the norm $\|u\|_{W^{k,p}(\om)} = \|u\|_{k,p} = \|u\|_p + \sum_{|\gamma|=k} \|\partial^\gamma u\|_p$. The second term in the norm will be denoted by $|u|_{W^{1,p}(\om)} = |u|_{k,p}$. Let $W_0^{-k,p'}(\om) = \left(W^{k,p}(\om)\right)'$ and $W^{-k,p'}(\om) = \left(W_0^{k,p}(\om)\right)'$. $W^{1-\frac{1}{p},p}(\pom)$ denotes the space of traces of $W^{1,p}(\om)$--functions on the boundary. It coincides with the Besov space $B^{1-\frac{1}{p}}_{p,p}(\pom)$ as obtained by real interpolation of Sobolev spaces \cite{tri}, and one may define $W^{s,p}(\pom)=B^{s}_{p,p}(\pom)$ for $s\in (-1,1)$.
\end{definition}

\begin{remark} \label{sobremarks}
We are going to need the following properties for bounded $\pom$ \cite{tri}: \\
a) All the above spaces are reflexive and $\left(W^{s,p}(\pom)\right)'=W^{-s,p'}(\pom)$.\\
b) For $p=2$ they coincide with the Sobolev spaces $H^s$.\\
c) $W^{1-\frac{1}{p},p}(\pom) \hookrightarrow W^{\frac{1}{2},2}(\pom)$ for $p \geq 2$.\\
d) If $\pom$ is smooth, pseudodifferential operators of order $m$ with symbol in the H\"ormander class $S^m_{1,0}(\pom)$ map $W^{s,p}(\pom)$ continuously to $W^{s-m,p}(\pom)$. For Lipschitz $\pom$, at least the first--order Steklov--Poincar\'e operator $S$ of the Laplacian on $\omc$ is continuous between $W^{\frac{1}{2},2}(\pom)$ and $W^{-\frac{1}{2},2}(\pom)$ \cite{cos}. \\
e) Points a) to d) imply that the quadratic form $\langle S u, u \rangle$ associated to $S$ is well-defined on $W^{1-\frac{1}{p},p}(\pom)$ if $p \geq 2$. $S$ being elliptic, the form cannot be defined for $p<2$ even if $\pom$ is smooth.
\end{remark}

Uniform monotony will be shown using a variant of Friedrichs' inequality.
\begin{proposition}\label{poinfried}
Assume $\om$ is bounded and that $\Gamma \subset \pom$ has positive $(n-1)$--dimensional measure. Then there is a $C>0$ such that $$\|u\|_p \leq C( \|\nabla u\|_p + \|u|_{\Gamma}\|_{L^1(\Gamma)}) \quad \text{for all $u \in W^{1,p}(\om)$}.$$
\end{proposition}
\begin{proof}
We apply an interpolation argument to the well-known Friedrichs' inequality $$\|u-u_\om\|_p \leq C \|\nabla u\|_p, \qquad u_\om = \frac{1}{|\om|}\int_\om u,$$ on $W^{1,p}(\om)$ (see e.g.~\cite{nec}). Let $L : W^{1,p}(\om) \to L^p(\om)$ be the rank--$1$ operator $Lu = \frac{1}{|\Gamma|}\int_\Gamma u|_\Gamma$ and $I$ the inclusion of $W^{1,p}(\om)$ into $L^p(\om)$. Then $I-L:W^{1,p}(\om) \to L^p(\om)$ is bounded and
$$\|u-Lu\|_p=\|(I-L)(u-u_\om)\|_p\leq\|I-L\| \|u-u_\om\|_{1,p} \leq C \|\nabla u\|_p$$
for all $u \in W^{1,p}(\om)$. The assertion follows.
\end{proof}

Let $\omega(x,y) = (|x|+|y|)^{\delta}(1+|x|+|y|)^{1-\delta}$, $0 \leq \delta \leq 1$. In addition to the above norms, the following family of quasi--norms will prove useful:

\begin{definition}
For $v,w \in W^{1,p}(\om)$ and $k \in \N_0$, define
$$|v|_{(k,w,p)} = \left(\int_\om \omega(\nabla w, D^k v)^{p-2} |D^k v|^2\right)^{\frac{1}{2}},$$
where $|D^k v|^2 = \sum_{|\gamma| = k} |\partial^\gamma v|^2$.
\end{definition}

\begin{remark} \label{quasiremark1}
a) If $p\geq 2$, the $(1,w,p)$--quasi--norm can be estimated from above and below by suitable powers of the $W^{1,p}$--seminorm \cite{el}:
$$|v|_{1,p}^p \leq |v|_{(1,w,p)}^2 \leq C(|v|_{1,p}, |w|_{1,p}) |v|_{1,p}^2.$$
b) In the nondegenerate case $\delta=0$, we have
$|v|_{1,2}^2 \leq |v|_{(1,w,p)}^2$.\\
c) The following inequality is useful for computations with quasi--norms:
$$\lambda \mu \leq \max\{\varepsilon^{-1}, \varepsilon^{1/(1-p)}\} (a^{p-1} + \lambda)^{p'-2} \lambda^2 + \varepsilon(a+\mu)^{p-2}\mu^2$$
for $\lambda, \mu, a \geq 0$ and $\varepsilon >0$.
\end{remark}

The results of this paper easily generalize to the systems of equations describing certain inelastic materials. In this case, Lemma~\ref{poinfried} has to be replaced by the following Korn inequality:
\begin{proposition}\label{korn}
Assume $\om \subset \R^n$ is a bounded Lipschitz domain and $\Gamma \subset \pom$ has positive $(n-1)$--dimensional measure. Then there is a $C>0$ such that $$\|u\|_{1,p} \leq C( \|\varepsilon(u)\|_p + \|u|_{\Gamma}\|_{L^1(\Gamma)}) \quad \text{for all $u \in (W^{1,p}(\om))^n$}.$$
\end{proposition}
\begin{proof}
The $L^p$--version $\|u\|_{1,p} \leq C( \|\varepsilon(u)\|_p + \|u\|_{p})$ of Korn's inequality is well-known (see e.g.~\cite{korn}). Assume the assertion was false. Then $\|\varepsilon(u_n)\|_p + \|u_n|_{\Gamma}\|_{L^1(\Gamma)} \leq \frac{1}{n}$ for some sequence
in $W^{1,p}(\om)$ normalized to $\|u_n\|_{1,p}=1$. By the compactness of $W^{1,p}(\om) \hookrightarrow L^p(\om)$, we may assume $u_n$ to converge in $L^p(\om)$. The cited variant of Korn's inequality shows that $u_n$ is even Cauchy in $W^{1,p}(\om)$, hence converges to some $u_0$ with $\|\varepsilon(u_0)\|_p = \|u_0|_{\Gamma}\|_{L^1(\Gamma)}=0$. The kernel of $\varepsilon$ consists of skew--symmetric affine transformations $A x + b$, $A = -A^T$. As $\dim \mathrm{ker}\, A \equiv n \,\, \mathrm{mod}\,2$, $u_0$ cannot vanish on all of the ($n-1$--dimensional) $\Gamma$ unless $u_0=0$. Contradiction to $\|u_0\|_{1,p}=1$.\end{proof}

\section{Variational Formulation and Reduction to $\pom$}\label{formulation}

We continue to use the notation from the Introduction and mainly follow \cite{mast}. Fix some $p\geq2$ and, for $q(t)= \int_0^t s \varrho(s)\ \mathrm{d}s$, let $G(u) = \int_\om q(|\nabla u|)$ with derivative $$DG(u,v)=\langle G' u, v\rangle = \int_\om \varrho(|\nabla u|) \nabla u \nabla v \qquad \text{($u, v \in W^{1,p}(\om)$)}$$ and $j(v) = \int_{\gs} g |v|$, $v \in L^1(\gs)$. $G$ is known to be strictly convex and $G' : W^{1,p}(\om) \to \left(W^{1,p}(\om)\right)'$ bounded and uniformly monotone, hence coercive, with respect to the seminorm $| \cdot |_{1,p}$: There is some $\alpha_G>0$ such that for all $u, v \in W^{1,p}(\om)$
$$\langle G' u - G' v, u-v\rangle \geq \alpha_G |u-v|_{1,p}^p  \quad \text{and} \quad   \lim_{|u|_{1,p} \to \infty} \frac{\langle G' u, u\rangle}{|u|_{1,p}} = \infty.$$

The naive variational formulation of the transmission problem (\ref{diff}) minimizes the functional
$$\Phi(u_1, u_2) = G(u_1) + \frac{1}{2} \int_\omc |\nabla u_2|^2-\int_\om f u_1 - \langle t_0, u_2|_\pom\rangle + j((u_2-u_1+u_0)|_\gs)$$
over a suitable convex set.

\begin{lemma}
Minimizing $\Phi$ over the nonempty, closed and convex subset
$$C = \{(u_1,u_2)\in W^{1,p}(\om) \times W^{1,2}_{loc}(\omc) : (u_1-u_2)|_\gt= u_0, \, u_2 \in \mathcal{L}_2\},$$
$$\mathcal{L}_2 = \{v \in W^{1,2}_{loc}(\omc): \Delta v = 0\, \text{in $W^{-1,2}(\omc)$ $+$ radiation condition at $\infty$}\},$$
is equivalent to the system (\ref{diff}) in the sense of distributions if $\varrho\in C^1(\overline{\Omega} \times (0,\infty))$.
\end{lemma}
\begin{proof}
$C$ is apparently convex. A similar argument as in Remarks 2 and 4 of \cite{cag} shows that $C$ is closed and nonempty. The proof there almost exclusively involves the exterior problem in $\mathcal{L}_2$ and only requires basic measure theoretic properties of $W^{1,2}(\om)$, which also hold for $W^{1,p}(\om)$. Finally, repeat the computations of \cite{mast} to obtain equivalence with (\ref{diff}).
\end{proof}

To reduce the exterior problem to the boundary, we are going to need
the layer potentials
\begin{eqnarray*}
\mathcal{V} \phi(x) &=& -\frac{1}{\pi} \int_\pom \phi(x')\ \log|x-x'|\ dx',\\
\mathcal{K} \phi(x) &=& -\frac{1}{\pi} \int_\pom \phi(x')\ \partial_{\nu_{x'}} \log|x-x'| \ dx',\\
\mathcal{K}'\phi(x) &=& -\frac{1}{\pi} \int_\pom \phi(x')\
\partial_{\nu_x} \log|x-x'|\ dx',\\
\mathcal{W} \phi(x) &=& \frac{1}{\pi}\
\partial_{\nu_x}\int_\pom \phi(x')\ \partial_{\nu_{x'}} \log|x-x'| \
dx'
\end{eqnarray*}
associated to the Laplace equation on $\omc$. They extend from
$C^\infty(\pom)$ to a bounded map $\begin{pmatrix} -\mathcal{K} & \mathcal{V}\\
\mathcal{W} & \mathcal{K}'
\end{pmatrix}$ on the Sobolev space $ W^{\frac{1}{2},2}(\pom) \times
W^{-\frac{1}{2},2}(\pom)$. If the capacity of $\pom$ is less than
$1$, which can always be achieved by scaling, $\mathcal{V}$ and
$\mathcal{W}$ considered as operators on $W^{-\frac{1}{2},2}(\pom)$
are selfadjoint, $\mathcal{V}$ is positive and $\mathcal{W}$
non-negative. Similarly, the Steklov-Poincar\'e operator
$$S = \mathcal{W}+(1-\mathcal{K}')\mathcal{V}^{-1}(1-\mathcal{K}): W^{\frac{1}{2},2}(\pom) \subset
W^{-\frac{1}{2},2}(\pom) \to W^{-\frac{1}{2},2}(\pom)$$ defines a
positive and selfadjoint operator (pseudodifferential of order $1$,
if $\pom$ is smooth) with the main property
$$\partial_\nu u_2|_\pom = - S (u_2|_\pom-a)$$ for solutions $u_2 \in \mathcal{L}_2$ of the Laplace equation on $\omc$.
By Remark~\ref{sobremarks} e), $S$ gives rise to a coercive and symmetric bilinear form $\langle S u , u\rangle$ on $W^{\frac{1}{2},2}(\pom)$ and, in particular, a pairing on the traces of $W^{1,p}(\om)$ if and only if $p \geq 2$. \\

Using the weak definition of $\partial_\nu|_\pom$, $S$ reduces the integral over $\omc$ in $\Phi$ to the boundary:
$$\int_\omc |\nabla u_2|^2 = - \langle \partial_\nu u_2|_\pom, u_2|_\pom\rangle = \langle S (u_2|_\pom-a), u_2|_\pom\rangle \quad \text{for $u_2 \in \mathcal{L}_2$.}$$
Easy manipulations allow to substitute $u_2$ by a function $v$ on $\gs$ (cf. \cite{mast}): Let
$$\Wgs = \{u \in W^{\frac{1}{2},2}(\pom) : \mathrm{supp}\ u \subset \bar{\Gamma}_s\},\quad  X^p = W^{1,p}(\om) \times \Wgs$$ and $(u,v) = (u_1 -c, u_0+u_2|_\pom-u_1|_\pom) \in X^p$ for a suitable $c \in \R$. Collecting the data--dependent terms in $$\lambda(u,v) = \langle t_0 +S u_0, u|_\pom + v\rangle + \int_\om f u$$ leads to
$$\Phi(u_1,u_2) =  G(u) + \frac{1}{2} \langle S(u|_\pom +v),u|_\pom +v\rangle - \lambda(u,v) +j(v) + \frac{1}{2} \langle Su_0, u_0\rangle + \langle t_0, u_0\rangle.$$ The first three terms on the right hand side will be called $J(u,v)$.
\begin{lemma}
Minimizing $\Phi$ over $C$ is equivalent to minimizing $J + j$ over the nonempty closed convex set
$D = \{(u,v) \in X^p : \langle S(u|_\pom +v-u_0), 1\rangle = 0 \,\, \text{if $n=2$}\}$
\end{lemma}
\begin{proof}
As in \cite{mast}. The main additional observation here is that the substitution $v = u_0+u_2|_\pom-u_1|_\pom$ indeed defines an element of $\Wgs$, because $u_0, u_2|_\pom \in W^{\frac{1}{2},2}(\pom)$, $u_1|_\pom \in W^{1-\frac{1}{p},p}(\pom) \subset W^{\frac{1}{2},2}(\pom)$ by Remark~\ref{sobremarks} and $v|_\gt = 0$, if $(u_1, u_2) \in C$.
\end{proof}

\section{Existence and Uniqueness}\label{sec:uniq}

Minimization of $J+j$ over $D$ translates into the following variational inequality: Find $(\hat u, \hat v) \in X^p$ such that
$$\langle G'\hat u, u-\hat u \rangle + \langle S(\hat u|_\pom+\hat v), (u-\hat u)|_\pom + v-\hat v\rangle + j(v)-j(\hat v) \geq \lambda(u-\hat u, v-\hat v)$$
for all $(u,v) \in X^p$. Note that $D$ has been replaced by $X^p$.\\

We now prove the crucial monotony estimate:
\begin{lemma}\label{monotony}
The operator in the variational inequality is uniformly monotone on $X^p$. There exists an $\alpha = \alpha(C) > 0$ such that for all $\|u,v\|_X, \|\hat u,\hat v\|_X <C$
\begin{align*}
&\alpha (\|u-\hat u\|^p_{W^{1,p}(\om)} + \|v-\hat v\|^p_{\widetilde
  W^{\frac{1}{2},2}(\gs)}) \leq \langle G' \hat u - G' u, \hat u -
u\rangle \\
 & \hspace*{2.5cm}+\
\langle S((\hat u-u)|_\pom+\hat v-v), (\hat u-u)|_\pom+\hat v-v\rangle.
\end{align*}
\end{lemma}
\begin{proof}
Recall the monotony estimate for $G'$ from Section~\ref{formulation}:
$$\langle G' \hat u - G' u, \hat u - u\rangle \geq \alpha_G |\hat u-u|_{1,p}^p.$$
The triangle inequality and convexity of $x^p$ imply
\begin{eqnarray*}
\|\hat v-v\|^p_{\widetilde W^{\frac{1}{2},2}(\gs)} &\leq& (\|(\hat u-u)|_\gs+\hat v-v\|_{ W^{\frac{1}{2},2}(\gs)}+\|(\hat u-u)|_\gs\|_{ W^{\frac{1}{2},2}(\gs)})^p\\
& \leq & 2^{p-1} \ (\|(\hat u-u)|_\gs+\hat v-v\|^p_{ W^{\frac{1}{2},2}(\gs)}+\|(\hat u-u)|_\gs\|^p_{ W^{\frac{1}{2},2}(\gs)}).
\end{eqnarray*}
Using $W^{1-\frac{1}{p},p}(\gs) \hookrightarrow W^{\frac{1}{2},2}(\gs)$ as well as the boundedness of the trace operator,
$$2^{1-p}\ \|\hat v-v\|^p_{\widetilde W^{\frac{1}{2},2}(\gs)} - \beta\ \|\hat u-u\|^p_{W^{1,p}(\om)} \leq \|(\hat u-u)|_\gs+\hat v-v\|^p_{ W^{\frac{1}{2},2}(\gs)}$$
follows for some $\beta\geq 1$. Let $$K = \{(u,v,\hat u,\hat v) \in X^p\times X^p : \|(\hat u-u)|_\pom+\hat v-v\|_{W^{\frac{1}{2},2}(\pom)}  < 2 \beta C\}$$
and $0<\varepsilon<\beta^{-1}$. Since $S$ is positive definite on $W^{\frac{1}{2},2}(\pom)$, we obtain from Friedrichs' inequality for $(u,v,\hat u, \hat v) \in K$ or, in particular, if $\|u,v\|_X, \|\hat u,\hat v\|_X <C$:
\begin{align*}
&\langle G' \hat u - G' u, \hat u - u\rangle + \langle S((\hat u-u)|_\pom+\hat v-v), (\hat u-u)|_\pom+\hat v-v\rangle\\
&\gtrsim |\hat u-u|_{1,p}^p  + \|(\hat u-u)|_\pom+\hat v-v\|^2_{W^{\frac{1}{2},2}(\pom)}\\
&\gtrsim |\hat u-u|_{1,p}^p + \|(\hat u-u)|_\pom+\hat v-v\|^p_{W^{\frac{1}{2},2}(\pom)}\\
&\gtrsim |\hat u-u|_{1,p}^p + \varepsilon\ \|(\hat u-u)|_\gs+\hat v-v\|^p_{ W^{\frac{1}{2},2}(\gs)} + \|(\hat u-u)|_\gt \|^p_{W^{\frac{1}{2},2}(\gt)}\\
&\gtrsim \|\hat u-u\|_{W^{1,p}(\om)}^p  + \varepsilon\ \|(\hat u-u)|_\gs+\hat v-v\|^p_{ W^{\frac{1}{2},2}(\gs)}\\
&\gtrsim (1-\varepsilon \beta)\ \|\hat u-u\|_{W^{1,p}(\om)}^p + 2^{1-p}{\varepsilon}\ \|\hat v-v\|^p_{\widetilde W^{\frac{1}{2},2}(\gs)}.
\end{align*}
Uniform monotony on all of $X^p$ is shown similarly, but on the unbounded complement $(X^p\times X^p) \setminus K$ the exponents $p$ on the left hand side have to be replaced by $2$.
\end{proof}
\begin{theorem}\label{exuniq}
The variational inequality is equivalent to the transmission problem (\ref{diff}) and has a unique solution.
\end{theorem}
\begin{proof}
We repeat the computations in \cite{mast} to get the equivalence with the minimization of $J+j$ over $D$, and hence with (\ref{diff}).
Existence and uniqueness follow from Lemma~\ref{monotony}, e.g.~by applying \cite{zei}, Proposition 32.36.
\end{proof}
\section{Discretization and Error Analysis}\label{sec:disc}

In order to avoid using
$S=\mathcal{W}+(1-\mathcal{K}')\mathcal{V}^{-1}(1-\mathcal{K})$
explicitly, the numerical implementation involves a variant of the
variational inequality
$$\langle G'\hat u, u-\hat u
\rangle + \langle S(\hat u|_\pom+\hat v), (u-\hat u)|_\pom + v-\hat
v\rangle + j(v)-j(\hat v) \geq \lambda(u-\hat u, v-\hat v)$$ in
terms of the layer potentials. Our a posteriori analysis is
therefore based on the following equivalent problem: Find $(\hat
u,\hat v, \hat \phi) \in X^p \times W^{-\frac{1}{2},2}(\pom)=: Y^p$,
such that
\begin{align*}
&\langle G'\hat u, u-\hat u \rangle +\langle \mathcal{W}(\hat
u|_\pom+\hat v) + (\mathcal{K}'-1) \hat \phi, (u-\hat u)|_\pom
+v-\hat v\rangle\\
&\hspace{0.8cm}+ j(v)-j(\hat v)\ \geq \
\langle
t_0 + \mathcal{W} u_0, (u-\hat u)|_\pom +v-\hat v\rangle + \int_\om f (u-\hat u), \\
 &\langle \phi, \mathcal{V} \hat \phi + (1-\mathcal{K})(\hat
u|_\pom+\hat v)\rangle = \langle \phi, (1-\mathcal{K}) u_0\rangle
\end{align*}
for all $(u,v,\phi) \in Y^p$. More concisely,
$$B(\hat u, \hat v, \hat \phi; u-\hat u, v-\hat v, \phi-\hat \phi)+ j(v)-j(\hat v) \geq \Lambda(u-\hat u, v-\hat v, \phi-\hat \phi)$$
with
\begin{eqnarray*}
B(u, v, \phi; \bar u, \bar v, \bar \phi)&=&\langle G'u, \bar u
\rangle +\langle \mathcal{W}(u|_\pom+v) + (\mathcal{K}'-1) \phi,
\bar u|_\pom +\bar
v\rangle\\
&&\quad+\langle \bar \phi, \mathcal{V} \phi +
(1-\mathcal{K})(u|_\pom+v)\rangle,\\
\Lambda(u, v, \phi) &=& \langle t_0 + \mathcal{W} u_0, u|_\pom
+v\rangle + \int_\om f u+\langle \phi, (1-\mathcal{K}) u_0\rangle.
\end{eqnarray*}

The more detailed a priori and a posteriori error analysis requires
a few basic properties of the quasi--norms \cite{el}.
\begin{remark}\label{quasiremark2}
a) The continuity and coercivity estimates can be sharpened: For all $u,v \in \Wp$
$$\langle G' u - G' v, u-v\rangle\lesssim |u-v|_{(1,u,p)}^2 \lesssim \langle G' u - G' v, u-v\rangle.$$
b) There is $\theta>0$ such that for all $\varepsilon \in
(0,\infty)$ and all $u,v,w \in \Wp$
$$|\langle G' u - G' v, w\rangle| \lesssim \varepsilon |u-v|_{(1,u,p)}^2 + \varepsilon^{-\theta} |w|^2_{(1,u,p)}.$$
\end{remark}

\begin{lemma} \label{Bcoercive}
For all $(\hat u,\hat v,\hat \phi), (u,v,\phi) \in Y^p$ we have
\begin{align*}
& |\hat u - u|_{(1,\hat u, p)}^2 + \|(\hat u - u)|_\pom + \hat v - v\|_\Wg^2 + \|\eta\|_{W^{-\frac{1}{2},2}(\pom)}^2 \\
& \lesssim |\hat u - u|_{(1,\hat u, p)}^2 + \|(\hat u - u)|_\pom + \hat v - v\|_\Wg^2 + \|\hat \phi- \phi\|_{W^{-\frac{1}{2},2}(\pom)}^2\\
& \lesssim  B(\hat u,\hat v,\hat \phi;\hat u- u, \hat v- v, \eta) -
B(u,v,\phi;\hat u- u, \hat v- v, \eta),
\end{align*}
where $2 \eta = \hat \phi -\phi + V^{-1}(1-K)((\hat u-u)|_\pom +
\hat v-v)$.
\end{lemma}
\begin{proof}
The right hand side of the identity
\begin{align*}
& B(\hat u,\hat v,\hat \phi;\hat u- u, \hat v- v,
\eta) -
B(u,v,\phi;\hat u- u, \hat v- v, \eta)\\
&= \langle G'\hat u - G'u, \hat u - u\rangle +
\textstyle{\frac{1}{2}} \langle \mathcal{W}((\hat u - u)|_\pom +
\hat v - v),(\hat u - u)|_\pom + \hat v - v)\rangle  \\
&\hphantom{=} +
\textstyle{\frac{1}{2}} \langle S((\hat u - u)|_\pom + \hat v -
v),(\hat u - u)|_\pom + \hat v - v)\rangle+
\textstyle{\frac{1}{2}}\langle \mathcal{V}(\hat \phi - \phi), \hat
\phi - \phi \rangle.
\end{align*}
is, up to a constant, larger than $\|\hat u - u,\hat v - v,\hat
\phi-\phi\|^2_{(\hat u, Y^p)}$. Furthermore,
$$\|\eta\|_{W^{-\frac{1}{2},2}(\pom)}\lesssim \|\hat \phi - \phi\|_{W^{-\frac{1}{2},2}(\pom)}+\|(\hat u - u)|_\pom + \hat v - v\|_{W^{\frac{1}{2},2}(\pom)}.$$
\end{proof}

Let $\{\mathcal{T}_h\}_{h\in I}$ a regular triangulation of
$\om$ into disjoint open regular triangles $K$, so that
$\overline{\om} = \bigcup_{K \in \mathcal{T}_h} K$. Each
element has at most one edge on $\pom$, and the closures of any two
of them share at most a single vertex or edge. Let $h_K$ denote the
diameter of $K \in \mathcal{T}_h$ and $\rho_K$ the diameter of the
largest inscribed ball. We assume that $1 \leq \max_{K \in
\mathcal{T}_h} \frac{h_K}{\rho_K} \leq R$ independent of $h$ and
that $h = \max_{K\in \mathcal{T}_h} h_K$. $\mathcal{E}_h$ is going
to be the set of all edges of the triangles in $\mathcal{T}_h$, $D$
the set of nodes. Associated to $\mathcal{T}_h$ is the space $\Whp \subset
\Wp$ of functions whose restrictions to any $K \in \mathcal{T}_h$
are linear.

$\pom$ is triangulated by $\{l \in \mathcal{E}_h : l
\subset \pom\}$. $\Whg$ denotes the corresponding space of piecewise
linear functions, and $\Whgs$ the subspace of those supported on
$\gs$. Finally, $\Whmg \subset \Wmg$.

We denote by $i_h: \Whp \hookrightarrow \Wp$, $j_h :
\Whgs\hookrightarrow \Wgs$ and $k_h: \Whmg \hookrightarrow \Wmg$ the
canonical inclusion maps. Set $X^p_h = \Whp \times \Whgs$, We denote
by $i_h: \Whp \hookrightarrow \Wp$, $j_h : \Whgs\hookrightarrow
\Wgs$ and $k_h: \Whmg \hookrightarrow \Wmg$ the canonical inclusion
maps. Set $X^p_h = \Whp \times \Whgs$,
\[
S_h = \frac12( W+(I-K')k_h(k_h^* V k_h)^{-1} k_h^*(I-K))
\]
and $$\lambda_h(u_h,v_h) =\langle t_0+ S_h u_0, u|_\pom + v\rangle +
\int_\om f u_h.$$
As is well--known, there exists $h_0>0$ such that
the approximate Steklov--Poincar\'e operator $S_h$ is coercive
uniformly in $h<h_0$, i.e.~$\langle S_h u_h, u_h\rangle \geq
\alpha_S \|u_h\|_{\Wg}^2$ with $\alpha_S$ independent of $h$.

The discretized variational inequality reads as follows: Find $(\hat
u_h, \hat v_h, \hat \phi_h) \in Y^p_h$ such that
$$B(\hat u_h, \hat v_h, \hat \phi_h; u_h-\hat u_h, v_h-\hat v_h, \phi_h-\hat \phi_h)+ j(v_h)-j(\hat v_h) \geq \Lambda(u_h-\hat u_h, v_h-\hat v_h, \phi_h-\hat \phi_h)$$
for all $(u_h, v_h, \phi_h) \in Y^p_h$.
Repeating the arguments from the previous section, one obtains a
unique solution to the discretized variational inequality.

\begin{theorem} \label{apriori}
Let $(\hat u, \hat v, \hat \phi) \in Y^p$, $(\hat u_h, \hat v_h,
\hat \phi_h) \in Y^p_h$ be the solutions of the continuous
resp.~discretized variational problem. The following a priori bound
for the error holds uniformly in $h<h_0$:
\begin{align*}
&\|\hat u - \hat u_h, \hat v - \hat v_h, \hat \phi
- \hat
\phi_h\|_{Y^p}^p\\
 &\lesssim |\hat u - \hat u_h|_{(1,\hat u, p)}^2 + \|(\hat u - \hat u_h)|_\pom + \hat v - \hat v_h\|_\Wg^2 + \|\hat \phi- \hat \phi_h\|_{W^{-\frac{1}{2},2}(\pom)}^2\\
&\lesssim \inf_{(u_h,v_h,\phi_h) \in Y^p_h} \|\hat u - u_h, \hat v
- v_h, \hat \phi - \phi_h\|_{Y^p}^2 + \|\hat v-v_h\|_{L^2(\gs)} .
\end{align*}
\end{theorem}
\begin{proof}
Let $(u,v,\phi)\in Y^p$, $(u_h,v_h,\phi_h)\in Y^p_h$. Lemma~\ref{Bcoercive} and the variational inequality imply
\begin{align*}
&|\hat u - \hat u_h|_{(1,\hat u, p)}^2 + \|(\hat u - \hat u_h)|_\pom + \hat v - \hat v_h\|_\Wg^2 + \|\hat \phi- \hat \phi_h\|_{W^{-\frac{1}{2},2}(\pom)}^2\\
&\lesssim B(\hat u, \hat v, \hat \phi; \hat u - \hat u_h, \hat v -
\hat v_h, \hat \phi-\hat \phi_h) - B(\hat u_h, \hat v_h, \hat
\phi_h;
\hat u - \hat u_h, \hat v - \hat v_h, \hat \phi-\hat \phi_h)\\
& \lesssim B(\hat u, \hat v, \hat \phi; u, v, \phi) - \Lambda(u-\hat
u, v-\hat v, \phi-\hat\phi) + j(v) - j(\hat v)\\
& \hphantom{\lesssim} +B(\hat u_h, \hat v_h, \hat \phi_h; u_h,v_h,\phi_h) -\Lambda(u_h-\hat u_h, v_h-\hat v_h, \phi_h-\hat \phi_h) + j(v_h) - j(\hat v_h)\\
& \hphantom{\lesssim} - B(\hat u_h, \hat v_h, \hat \phi_h; \hat u, \hat v, \hat
\phi) - B(\hat u, \hat v, \hat \phi; \hat u_h, \hat v_h, \hat
\phi_h)
\end{align*}
Setting $(u,v,\phi)=(\hat u_h, \hat v_h, \hat \phi_h)$ and adding
$0$, the right hand side turns into
\begin{align*}
& B(\hat u, \hat v, \hat \phi; u_h-\hat u, v_h-\hat
v, \phi_h-\hat\phi) - \Lambda(u_h-\hat
u, v_h-\hat v, \phi_h-\hat\phi) + j(v_h) - j(\hat v)\\
&  \quad +B(\hat u, \hat v, \hat \phi; \hat u-u_h, \hat v-v_h,
\hat\phi-\phi_h)-B(\hat u_h, \hat v_h, \hat \phi_h; \hat u-u_h,\hat
v-v_h,\hat \phi-\phi_h).
\end{align*}
We first consider the friction terms:
$$j(v_h) - j(\hat v) = \int_\gs g(|v_h| - |\hat v|) \leq \int_\gs g(|v_h - \hat v|) \leq \|g\|_{L^2(\gs)}\|v_h - \hat v\|_{L^2(\gs)}.$$
The last two terms are bounded using Remark~\ref{quasiremark2}b and
Cauchy-Schwarz:
\begin{eqnarray*}
\langle G'\hat u - G'\hat u_h, \hat u - u_h \rangle &\lesssim&
\varepsilon |\hat u-\hat u_h|_{(1,\hat u,p)}^2 +
\varepsilon^{-\theta} |\hat u-u_h|^2_{(1,\hat u,p)},\\ & \lesssim &
\varepsilon |\hat u_h - \hat u|_{(1,\hat u, p)}^2 +
\varepsilon^{-\theta} C(|\hat u|_{1,p}, |u_h|_{1,p}) |u_h-\hat
u|_{1,p}^2
\end{eqnarray*}
for sufficiently small $\varepsilon >0$. We may replace $C(|\hat
u|_{1,p}, |u_h|_{1,p})$ by an honest constant noting that the
coercivity of our functional gives an a priori bound on $\|\hat
u\|_{\Wp}$ and that we can restrict to those $u_h$ satisfying
$\|u_h\|_{\Wp} \leq 2 \|\hat u\|_{\Wp}$. Moreover,
\begin{align*}
& \langle \mathcal{W}((\hat u - \hat u_h)|_\pom +
\hat v -
\hat v_h) + (1-\mathcal{K}') (\hat \phi-\hat \phi_h), (\hat u - u_h)|_\pom + \hat v - v_h\rangle\\
& \lesssim \varepsilon \|(\hat u - \hat u_h|_\pom + \hat v - \hat
v_h\|^2_{\Wg} +\varepsilon \|\hat \phi-\hat
\phi_h\|_{W^{-\frac{1}{2},2}(\pom)}^2 \\
&\hphantom{\lesssim} + \varepsilon^{-1} \|\hat u - u_h)\|^2_{\Wg} + \varepsilon^{-1}
\|\hat v - v_h\|^2_{\Wg},
\end{align*}
and
\begin{align*}
&\langle \hat \phi - \phi_h, \mathcal{V}(\hat \phi
- \hat \phi_h) + (1-\mathcal{K})((\hat u - \hat u_h)|_\pom + \hat v
- \hat
v_h)\rangle\\
&\lesssim \varepsilon^{-1}\|\hat \phi -
\phi_h\|^2_{W^{-\frac{1}{2},2}(\pom)} +\varepsilon \|\hat \phi -
\hat \phi_h\|^2_{W^{-\frac{1}{2},2}(\pom)}+ \varepsilon \|(\hat u -
\hat u_h)|_\pom + \hat v - \hat v_h\|^2_{\Wg}.
\end{align*}
Substituting $(u,v, \phi) = (u_h, \hat v,0)$ and $(u,v,\phi)=(2 \hat
u - u_h, \hat v, 0)$ into the variational inequality on $Y^p$ and
using that also the $\phi$ part is really an equality, the remaining
two terms reduce to
\begin{align*}
&\langle-t_0 - \mathcal{W} u_0+ \mathcal{W}(\hat
u|_\pom+\hat v) +
(\mathcal{K}'-1) \hat \phi, v_h - \hat v\rangle \\
&= - \langle t_0 - S(\hat u|_\pom + \hat v-u_0), v_h - \hat
v\rangle \\
&= - \langle \varrho(|\nabla u|)
\partial_\nu u, v_h - \hat v\rangle \ \leq\ \|g\|_{L^2(\gs)}\|v_h -
\hat v\|_{L^2(\gs)}.
\end{align*}
Applying these various estimates to the terms of the right hand
side, the assertion follows from
$$\|\hat u - \hat u_h,\hat v - \hat v_h, \hat \phi - \hat \phi_h\|_{Y^p}^p \lesssim |\hat u - \hat u_h|_{(1,\hat u, p)}^2 + \|(\hat u - \hat u_h)|_\pom + \hat v - \hat v_h\|_\Wg^2+\|\hat \phi - \hat \phi_h\|^2_{W^{-\frac{1}{2},2}(\pom)}$$
as in Lemma~\ref{monotony}.
\end{proof}

In the nondegenerate case $\delta = 0$, we essentially recover the estimates for uniformly elliptic operators from \cite{cag,mast}.

\begin{corollary} \label{ellapriori}
For $\delta = 0$, we obtain
$$\|\hat u - \hat u_h, \hat v - \hat v_h,\hat \phi - \hat
\phi_h\|_{Y^2}^2\ \lesssim \ \inf_{(u_h,v_h, \phi_h) \in Y^p_h}
\|\hat u - u_h, \hat v - v_h, \hat \phi - \phi_h\|_{Y^p}^2 + \|\hat
v-v_h\|_{L^2(\gs)}$$ uniformly in $h<h_0$
\end{corollary}
\begin{proof}
Use~\ref{quasiremark1}b) to estimate $|\hat u_h-\hat u|_{(1,\hat u,p)}$ in Theorem~\ref{apriori} from below.
\end{proof}
\section{A posteriori error estimate}\label{sec:post}

Denote by $$(e, \tilde e, \epsilon) = (\hat u - \hat u_h, \hat v -
\hat v_h, \hat \phi - \hat \phi_h) \in Y^p$$ the error of the
Galerkin approximation, and let
$2\nu=\epsilon+\mathcal{V}^{-1}(1-\mathcal{K})(e|_\pom+\tilde e)$.
Our basic a posteriori estimate is the following.
\begin{lemma} \label{abstractapost}
For all $(e_h, \tilde e_h, \nu_h) \in Y^p_h$
\begin{align*}
&|e|_{(1,\hat u, p)}^2 + \|e|_\pom + \tilde e\|_\Wg^2 + \|\epsilon\|_{W^{-\frac{1}{2},2}(\pom)}^2 \\
&\lesssim \Lambda(e - e_h,\tilde e - \tilde e_h, \nu-\nu_h) +
j(\tilde e_h +\hat v_h) - j(\hat v) \\
 &\hphantom{\lesssim}-B(\hat u_h, \hat v_h, \hat
 \phi_h; e
-e_h, \tilde e - \tilde e_h, \nu-\nu_h)\\
&= \int_\om f (e - e_h)-\langle G'\hat u_h, e - e_h\rangle+\int_\gs
g (|\tilde e_h +\hat v_h| - |\tilde e
+\hat v_h|)\\
 &\hphantom{\lesssim} -\langle \nu-\nu_h, \mathcal{V} \hat
\phi_h + (1-\mathcal{K})(\hat u_h|_\pom+\hat v_h-u_0)\rangle\\
&\hphantom{\lesssim}+\langle t_0-\mathcal{W}(\hat u_h|_\pom+\hat v_h-u_0) -
(\mathcal{K}'-1) \hat \phi_h, (e -e_h)|_\pom +\tilde e - \tilde
e_h\rangle.
\end{align*}
\end{lemma}
\begin{proof}
Lemma~\ref{Bcoercive}, the continuous and the discretized
variational inequality imply
\begin{align*}
&|\hat u - u|_{(1,\hat u, p)}^2 + \|(\hat u - u)|_\pom + \hat v - v\|_\Wg^2 + \|\hat \phi- \phi\|_{W^{-\frac{1}{2},2}(\pom)}^2 \\
& \lesssim  B(\hat u,\hat v,\hat \phi;\hat u- \hat u_h, \hat v-
\hat v_h, \nu) - B(\hat u_h,\hat v_h,\hat \phi_h;\hat u- \hat u_h,
\hat v- \hat v_h, \nu)\\
&\lesssim \Lambda(\hat u- \hat u_h, \hat v- \hat v_h, \nu) +j(\hat
v_h)-j(\hat v)-B(\hat u_h,\hat v_h,\hat \phi_h;\hat u- \hat u_h,
\hat v- \hat v_h, \nu)\\
& \lesssim  \Lambda(\hat u - \hat u_h - (u_h-\hat u_h),\hat v -
\hat v_h - (v_h-\hat v_h), \nu-\nu_h) + j(v_h)-j(\hat v) \\
 &\hphantom{\lesssim} - B(\hat u_h, \hat v_h, \hat \phi_h; \hat u - \hat u_h
-(u_h-\hat u_h), \hat v - \hat v_h - (v_h-\hat v_h), \nu-\nu_h).
\end{align*}
Note that the variational inequalities are identities when
restricted to the $\phi$-variable. The claim follows by setting $e_h
= u_h - \hat u_h$ and $\tilde e_h = v_h - \hat v_h$.
\end{proof}
Simplifying the right hand side along the lines of \cite{cly} leads to a gradient recovery scheme in the interior with a residual type estimator on the boundary. With a straight forward modification of \cite{ly}, also a method purely based on residual type estimates could be justified.

For $1< p<\infty$ and $0 \leq \delta \leq 1$, define $$G_{p,\delta}(x,y) = |y|^2 \omega(x,y)^{p-2}=|y|^2[(|x|+|y|)^{\delta}(1+|x|+|y|)^{1-\delta}]^{p-2}$$
whenever $|x|+|y|>0$ and $0$ otherwise.
As in \cite{cly}, our analysis will be based on the following consequences of the monotony and convexity properties of $G_{p,\delta}$.

\begin{lemma} \label{gpoinc}
Assume that $\om$ is connected. Let $q$ be a continuous linear form on $W^{1,p}(\om)$ with $\R \cap \ker q = \{0\}$, where $\R$ is identified with the space of constant functions on $\om$. Then for any $1<p<\infty$ there exists $C_P = C_P(p,q,\om)>0$ such that for all $a\geq 0$ and $u \in W^{1,p}(\om)$,
$$\int_\om G_{p,\delta}(a,u) \leq C_P \left(G_{p,\delta}(a, q(u)) + \int_\om G_{p,\delta}(a,|\nabla u|)\right).$$
\end{lemma}
\begin{proof}
Cf. \cite{cly}, Lemma 4.1 and its generalization in Remark 4.3.
\end{proof}

\begin{lemma} \label{combilemma}
For any $d, k \in \N$ there is $C_\Sigma=C_\Sigma(p,d,k)>0$ such that for all $a_1, a_2, \dots , a_k \in \R^d$
$$\sum_{j=1}^k\sum_{l=1}^{j-1} G_{p,\delta}(a_j, a_j-a_l) \lesssim C_\Sigma \sum_{j=1}^{k-1} \min_{1 \leq m \leq k} G_{p,\delta}(a_m, a_{j+1} - a_j).$$
\end{lemma}
\begin{proof}
Cf. \cite{cly}, Lemma 4.2 and its generalization in Remark 4.3.
\end{proof}

Even though Lemma~\ref{gradlemma} and Lemma~\ref{flemma} hold for any $1<p<\infty$ with minor modifications of the proofs (see \cite{cly} for a similar discussion), we will from now on concentrate on the range $2 \leq p <\infty$ relevant to our transmission problem.

\begin{definition}
Let $z \in D$ be a node of the triangulation $\mathcal{T}_h$ and
$\varphi_z \in \Whp$ the associated nodal basis function. Let
$\omega_z = \{x \in \om : \varphi_z(x) >0\}$ be the interior of the
support of $\varphi_z$. The interpolation operator $\pi : \Wp \to
\Whp$ is defined as
$$ \pi u = \sum_{z \in D} u_z \varphi_z, \qquad u_z = \int_\om \varphi_z u / \int_\om \varphi_z.$$
\end{definition}

\begin{lemma}\label{gradlemma}
Let $\mathcal{E}_h^z = \{l \in \mathcal{E}_h: l=\bar K_i \cap \bar
K_j \text{ for some } K_i, K_j\subset \omega_z\}$. Given $u_h \in
\Whp$, let $\left[\partial_{\nu_\mathcal{E}} u_h\right]_l$ denote
the jump of the normal derivative across the inner edge $l$ of the
triangulation. Then, if $v \in \Wp$ and $K \in \mathcal{T}_h$, the
following estimate holds:
\begin{align*}
&\int_K G_{p,\delta}(\nabla u_h, h_K^{-1}(v-\pi v)) + \int_K G_{p,\delta}(\nabla u_h, \nabla(v-\pi v)) \\
&\lesssim\sum_{z \in D \cap \bar K} \left(\int_{\omega_z}
G_{p,\delta}(\nabla u_h, \nabla v) +  \sum_{l \in \mathcal{E}_h^z}
\min_{\bar K' \cap l \neq \emptyset} \int_{\omega_z}
G_{p,\delta}(\nabla u_h|_{K'}, \left[\partial_{\nu_\mathcal{E}}
u_h\right]_l)\right).
\end{align*}
\end{lemma}
\begin{proof}
The proof is a modification of \cite{cly}, Lemma 4.3. Concerning the
first term on the left hand side, the convexity of $G_{p,\delta}$ in
its second argument (a ``triangle inequality'') and enlarging the
domain of integration leads to
\begin{eqnarray*}\int_K G_{p,\delta}(\nabla u_h, h_K^{-1}(v-\pi v))
& = & \int_K G_{p,\delta}(\nabla u_h, \sum_{z \in D \cap \bar K}h_K^{-1}(v- v_z)\varphi_z)\\
& \lesssim & \sum_{z \in D \cap \bar K} \int_K G_{p,\delta}(\nabla u_h, h_K^{-1}(v- v_z)\varphi_z)\\
& \leq & \sum_{z \in D \cap \bar K} \int_{\omega_z}
G_{p,\delta}(\nabla u_h|_K, h_K^{-1}(v- v_z)\varphi_z).
\end{eqnarray*}
As $G_{p,\delta}(\nabla u_h|_K, \cdot)$ is increasing and
$|\varphi_z|\leq 1$, Lemma~\ref{gpoinc} with $q(u) = \int_{\omega_z}
\varphi_z u$ implies
\begin{eqnarray}
\int_{\omega_z} G_{p,\delta}(\nabla u_h|_K, h_K^{-1}(v- v_z)\varphi_z)
&\leq & \int_{\omega_z} G_{p,\delta}(\nabla u_h|_K, h_K^{-1}(v- v_z)) \nonumber\\
&\leq&  C_P \int_{\omega_z} G_{p,\delta}(\nabla u_h|_K, \nabla(v- v_z)) \label{helpfulestimate1}\\
& = & C_P \int_{\omega_z} G_{p,\delta}(\nabla u_h|_K, \nabla v)
\nonumber
\end{eqnarray}
for every term in the sum over $z \in D \cap \bar K$. To replace the
constant $\nabla u_h|_K$ by $\nabla u_h$, we repeatedly apply the
usual triangle inequality and the convexity of $G_{p,\delta}$ to
obtain
\begin{align*}
&G_{p,\delta}(\nabla u_h|_K, \nabla v) \\
&\leq G_{p,\delta}(\nabla u_h|_K, |\nabla v| + |\nabla u_h|_K-\nabla u_h|)\\
&= (|\nabla v| + |\nabla u_h|_K-\nabla u_h|)^2 (|\nabla u_h|_K|+|\nabla v| + |\nabla u_h|_K-\nabla u_h|)^{\delta(p-2)}\\
&\hphantom{=} \times (1+|\nabla u_h|_K|+|\nabla v| + |\nabla u_h|_K-\nabla u_h|)^{(1-\delta)(p-2)}\\
&\leq (|\nabla v| + |\nabla u_h|_K-\nabla u_h|)^2 (|\nabla v| + 2(|\nabla u_h|+|\nabla u_h|_K-\nabla u_h|))^{\delta(p-2)}\\
&\hphantom{=}  \times (1+|\nabla v| + 2(|\nabla u_h|+|\nabla u_h|_K-\nabla u_h|))^{(1-\delta)(p-2)}\\
&\lesssim  G_{p,\delta}(\nabla u_h, |\nabla v| + |\nabla u_h|_K-\nabla u_h|)\\
&\lesssim  G_{p,\delta}(\nabla u_h, \nabla v) + G_{p,\delta}(\nabla
u_h,\nabla u_h|_K-\nabla u_h).
\end{align*}
Altogether
\[
 \int_K G_{p,\delta}(\nabla u_h|_K, h_K^{-1}(v-\pi v))
 \lesssim \sum_{z \in D \cap \bar K}
 \int_{\omega_z}\left\{G_{p,\delta}(\nabla u_h, \nabla v) +
   G_{p,\delta}(\nabla u_h,\nabla u_h|_K-\nabla u_h)\right\}.
\]
Let $\overline{\omega}_z = \bar K_1 \cup \cdots \cup \bar K_k$.
Applying Lemma~\ref{combilemma} with $a_j = \nabla u_h|_{K_j}$, $1
\leq j \leq k$, leads to the asserted bound for the first term. For
the proof, note that the conormal derivatives of the piecewise
linear function $u_h$ are determined by its boundary values on the
corresponding edge. But $u_h \in \Whp \subset \Wp$, so the
restrictions from both sides have to coincide, and the conormal
derivative does not jump: $a_j - a_{j-1} =
[\partial_{\nu_\mathcal{E}} u_h|_{\bar K_j \cap \bar K_{j-1}}]$.

As for the second term, let $c = \frac 1 {|K|} \int_K v$. Because
\begin{eqnarray*}
\int_{K} G_{p,\delta}(\nabla u_h, \nabla(v - \pi v)) \lesssim \int_K
G_{p,\delta}(\nabla u_h, \nabla v) + \int_K G_{p,\delta}(\nabla u_h,
\nabla (\pi v-c))
\end{eqnarray*}
by convexity and the triangle inequality, it only remains to
consider the second term $\int_K G_{p,\delta}(\nabla u_h, \nabla
(\pi v-c))$. The inverse estimate
$$|\nabla(\pi v - c)|  \lesssim \frac 1 {|K|} \int_K h_K^{-1} |\pi v - c|$$
for the affine function $\pi v - c$ and Jensen's inequality show
\begin{eqnarray*}\int_K G_{p,\delta}(\nabla u_h, \nabla (\pi v-c)) & \lesssim& \int_K \frac 1 {|K|} \int_K G_{p,\delta}(\nabla u_h, h_K^{-1}(\pi v-c))\\ &=& \int_K G_{p,\delta}(\nabla u_h, h_K^{-1}(\pi v-c)).\end{eqnarray*}
However, as before $$\int_K G_{p,\delta}(\nabla u_h, h_K^{-1}(\pi
v-c)) \lesssim \int_K G_{p,\delta}(\nabla u_h, h_K^{-1}(v-\pi v)) +
\int_K G_{p,\delta}(\nabla u_h, h_K^{-1}(v-c)),$$ and the first term
has been considered in the first step of the proof. Lemma~\ref{gpoinc} with $q(u) = \int_K u$ also bounds the final term by
$\int_K G_{p,\delta}(\nabla u_h, \nabla v)$.
\end{proof}

\begin{lemma} \label{flemma}
For any $\varepsilon >0$, $u_h \in \Whp$, $v \in \Wp$ and $f \in
L^{p'}(\om)$, 
\begin{eqnarray*}
\int_\om f (v - \pi v) &\leq& C \varepsilon \int_\om G_{p,\delta}(\nabla u_h, \nabla v)\\
& & + C(\varepsilon) \sum_{z \in D} \sum_{K \subset \overline{\omega}_z}\int_{K} G_{p',1}(|\nabla u_h|^{p-1}, h_K(f-f_K))\\
& & + C \varepsilon \sum_{z\in D}\sum_{l \in \mathcal{E}_h^z}
\min_{\bar K' \cap l \neq \emptyset} \int_{\omega_z}
G_{p,\delta}(\nabla u_h|_{K'}, \left[\partial_{\nu_\mathcal{E}}
u_h\right]_l).
\end{eqnarray*}
Here, $f_K = \frac{1}{|K|} \int_{K} f$. If $f \in W^{1,p'}(\om)$,
the second term may be replaced by
$$C(\varepsilon) \sum_{z \in D} \sum_{K \subset \overline{\omega}_z}\int_{K} G_{p',1}(|\nabla
u_h|^{p-1}, h_K^2\nabla f).$$
\end{lemma}
\begin{proof}
We adapt the proof of \cite{cly}, Lemma 4.4. Let $\tilde{K} \subset
\overline{\omega}_z$ such that $|\nabla u_h|_{\tilde{K}}| =
\max_{K'\subset \overline{\omega}_z} |\nabla u_h|_{K'}|$. Applying
the inequality from Remark~\ref{quasiremark1}c) for some
$\varepsilon >0$ and $C(\varepsilon) = C_P \max\{\varepsilon^{-1},
\varepsilon^{1/(1-p)}\}$,
\begin{eqnarray*}
\int_\om f (v - \pi v) &=& \sum_{z \in D} \sum_{K \subset \overline{\omega}_z}\int_{K} h_K (f-f_K) h_K^{-1} (v - v_z) \varphi_z\\
&\leq& C_P^{-1}\ C(\varepsilon)\sum_{z \in D} \sum_{K \subset \overline{\omega}_z}\int_{K}(|\nabla u_h|_{\tilde{K}}|^{p-1} + h_K |f-f_K|)^{p'-2} h_K^2 |f-f_K|^2 \\
& &+ \varepsilon \sum_{z \in D}\sum_{K \subset \overline{\omega}_z}\int_{K}(|\nabla u_h|_{\tilde{K}}|+h_K^{-1} |v - v_z| \varphi_z)^{p-2}h_K^{-2} |v - v_z|^2 \varphi_z^2\\
&\leq & C_P^{-1}\ C(\varepsilon)\sum_{z \in D}\sum_{K \subset \overline{\omega}_z}\int_{K}G_{p',1}(|\nabla u_h|_{\tilde{K}}|^{p-1}, h_K (f-f_K)) \\
& &+ \varepsilon \sum_{z \in D}\sum_{K \subset
\overline{\omega}_z}\int_{K} G_{p,\delta}(\nabla u_h|_{\tilde{K}},
h_K^{-1} (v - v_z) \varphi_z),
\end{eqnarray*}
because $\sum_{K \subset \overline{\omega}_z}\int_{K} f_K (v - v_z)
\varphi_z=0$. However, by our choice of $\tilde{K}$ and because
$p'\leq 2$,
$$\int_{K}G_{p',1}(|\nabla u_h|_{\tilde{K}}|^{p-1}, h_K (f-f_K)) \leq \int_{K}G_{p',1}(|\nabla u_h|^{p-1}, h_K (f-f_K)).$$
If $f \in W^{1,p'}(\om)$, Lemma~\ref{gpoinc} with $q(u) = \int_{K}
u$ gives:
$$\int_{K}G_{p',1}(|\nabla u_h|^{p-1}, h_K (f-f_K))
\leq  C_P \int_{K}G_{p',1}(|\nabla u_h|^{p-1}, h_K^2 \nabla f).$$
Concerning the $G_{p,\delta}$--term, equation
(\ref{helpfulestimate1}) in the proof of Lemma~\ref{gradlemma} shows
that it is dominated by $\varepsilon \int_{\omega_z}
G_{p,\delta}(\nabla u_h|_{\tilde{K}}, \nabla v)$, which in turn was
bounded by
$$\varepsilon \int_{\omega_z}
G_{p,\delta}(\nabla u_h, \nabla v) +  \varepsilon \sum_{l \in
\mathcal{E}_h^z} \min_{\bar K' \cap l \neq \emptyset}
\int_{\omega_z} G_{p,\delta}(\nabla u_h|_{K'},
\left[\partial_{\nu_\mathcal{E}} u_h\right]_l).$$
\end{proof}

In order to define the a posteriori estimator, we still need to
introduce some notation. For any $z \in D$, denote by $K_{j,z} \in
\mathcal{T}_h$, $1\leq j \leq N_z$, the triangles neighboring $z$ in
the sense that $\overline{\omega}_z = \bigcup_{j=1}^{N_z} \bar
K_{j,z}$. To each $K_{j,z}$ we associate a weight factor
$\alpha_{j,z} \geq 0$ normalized to $\sum_{j=1}^{N_z} \alpha_{j,z} =
1$.

\begin{definition}
Given $u_h \in \Whp$, define the gradient recovery $$G_h u_h =
\sum_{z \in D} (G_h v_h)(z)\ \varphi_z, \quad (G_h v_h)(z) =
\sum_{j=1}^{N_z} \alpha_{j,z} \nabla u_h|_{K_{j,z}}.$$
\end{definition}

The following theorem states our reliable, but presumably not efficient a posteriori estimate.

\begin{theorem}
Let $f \in L^{p'}(\om)$ and denote by $(e, \tilde e, \epsilon)$ the
error between the Galerkin solution $(\hat u_h, \hat v_h, \hat
\phi_h) \in Y^p_h$ and the true solution $(\hat u , \hat v, \hat
\phi) \in Y^p$. If $\gs \neq \emptyset$, assume that $\nabla \hat
u|_\gs \in L^p(\gs)$. Then
\begin{eqnarray*}
\|e, \tilde e, \epsilon\|_{Y^p}^p &\lesssim&
|e|_{(1,\hat u, p)}^2 + \|e|_\pom + \tilde e\|_\Wg^2 + \|\epsilon\|_{W^{-\frac{1}{2},2}(\pom)}^2 \\
&\lesssim& \eta_{gr}^2 + \eta_f^2 +\eta_S^2+\eta_\partial^2+\eta_g^2 ,
\end{eqnarray*}
where
\begin{align*}
\eta_{gr}^2 &= \sum_{K \in \mathcal{T}_h} \int_{K}G_{p,\delta}(\nabla \hat u_h, \nabla\hat u_h - G_h \hat u_h),\\
\eta_f^2 &= \sum_{K \in \mathcal{T}_h} \int_{K} G_{p',1}(|\nabla \hat u_h|^{p-1}, h_K (f-f_K)),\\
\eta_S^2 &= \sum_{l \subset \pom} l\ \|\partial_s \{\mathcal{V}
\hat \phi_h + (1-\mathcal{K})(\hat u_h|_\pom+\hat
v_h-u_0)\}\|_{L^2(l)}^2\\
\eta_{\partial}^2&=\sum_{l \subset \pom} l\ \|-\varrho(\nabla \hat u_h)\ \partial_{\nu} \hat
u_h+t_0-\mathcal{W}(\hat u_h|_\pom+\hat v_h-u_0) - (\mathcal{K}'-1)
\hat \phi_h\|_{L^2(l)}^2\\
\eta_{g}^2 &=\sum_{l\subset\gs} l
   \|\varrho(\nabla \hat u_h)\ \partial_{\nu} \hat
   u_h|_\gs\|^2_{L^2(l)}+\|g\|^2_{W^{-\frac{1}{2},2}(\gs)}
\end{align*}
If $f \in W^{1,p'}(\om)$, we may replace $\eta_f^2$ by $\sum_{K \in
\mathcal{T}_h} \int_{K} G_{p',1}(|\nabla \hat u_h|^{p-1}, h_K^2 \nabla
f)$.
\end{theorem}
\begin{proof}
From Lemma~\ref{abstractapost} we know that for all $(e_h, \tilde
e_h, \nu_h) \in Y^p_h$
\begin{align*}
&\|e, \tilde e, \epsilon\|_{Y^p}^p \lesssim  |e|_{(1,\hat u, p)}^2 + \|e|_\pom + \tilde e\|_\Wg^2 + \|\epsilon\|_{W^{-\frac{1}{2},2}(\pom)}^2 \\
&\lesssim  \int_\om f (e - e_h)- \sum_{K \in \mathcal{T}_h}
\int_{\partial K} \varrho(\nabla \hat u_h)\  \partial_\nu \hat u_h|_{\partial
K}\ (e-e_h)\\
&\hphantom{\lesssim} +\int_\gs g (|\tilde e_h +\hat v_h| - |\hat v|)-\langle
\nu-\nu_h, \mathcal{V} \hat
\phi_h + (1-\mathcal{K})(\hat u_h|_\pom+\hat v_h-u_0)\rangle\\
&\hphantom{\lesssim}+\langle t_0-\mathcal{W}(\hat u_h|_\pom+\hat v_h-u_0) -
(\mathcal{K}'-1) \hat \phi_h, (e -e_h)|_\pom +\tilde e - \tilde
e_h\rangle,
\end{align*}
with $2\nu=\epsilon+\mathcal{V}^{-1}(1-\mathcal{K})(e|_\pom+\tilde e)$.
The first two terms are mainly going to give the gradient recovery
in the interior, the fourth term the error $\eta_S$ of constructing
the Steklov-Poincar\'e operator, while the remaining terms add up to
$\eta_\partial$.

Concerning the first term:
\begin{eqnarray*}
\int_\om f(e - e_h) &\lesssim& \varepsilon \sum_{K \in \mathcal{T}_h} \int_{K} G_{p,\delta}(\nabla \hat u_h, \nabla e)\\
& & + C(\varepsilon)\sum_{K \in \mathcal{T}_h} \int_{K} G_{p',1}(|\nabla \hat u_h|^{p-1}, h_z (f-f_z))\\
& & + \varepsilon \sum_{z\in D}\sum_{l \in \mathcal{E}_h^z}
\min_{\bar K' \cap l \neq \emptyset} \int_{\omega_z}
G_{p,\delta}(\nabla \hat u_h|_{K'}, \left[\partial_{\nu_\mathcal{E}}
u_h\right]_l)\\
&\lesssim& \varepsilon |e|_{(1,\hat u, p)}^2+C(\varepsilon)\
\eta_f^2 + \varepsilon \sum_{z\in D}\sum_{l \in \mathcal{E}_h^z}
\min_{\bar K' \cap l \neq \emptyset} \int_{\omega_z}
G_{p,\delta}(\nabla \hat u_h|_{K'}, \left[\partial_{\nu_\mathcal{E}}
u_h\right]_l).
\end{eqnarray*}
$G_h \hat u_h$ is continuous across any interior edge $l$, so that
$[\partial_{\nu}\hat u_h]_l = [\partial_{\nu}\hat u_h - G_h \hat
u_h]_l$ and
$$\min_{\bar K' \cap l \ne \emptyset} \int_{\omega_z} G_{p,\delta}(\nabla \hat u_h|_{K'}, [\partial_{\nu}\hat u_h - G_hu_h]_l) \lesssim \int_{\omega_z} G_{p,\delta}(\nabla \hat u_h, \nabla \hat u_h - G_h \hat u_h).$$
Therefore,
\begin{eqnarray*}
\int_\om f(e - e_h) &\lesssim& \varepsilon |e|_{(1,\hat u,
p)}^2+C(\varepsilon) \eta_f^2 + \varepsilon \sum_{z\in D}\sum_{l \in
\mathcal{E}_h^z} \int_{\omega_z} G_{p,\delta}(\nabla \hat u_h,
[\partial_{\nu}\hat u_h- G_h \hat
u_h]_l)\\
&\lesssim& \varepsilon |e|_{(1,\hat u, p)}^2+C(\varepsilon) \eta_f^2 +\varepsilon \sum_{K \in \mathcal{T}_h} \int_{K} G_{p,\delta}(\nabla \hat u_h, \nabla \hat u_h - G_h \hat u_h)\\
&=&\varepsilon |e|_{(1,\hat u, p)}^2+C(\varepsilon) \eta_f^2 +\varepsilon \eta_{gr}^2.
\end{eqnarray*}
Concerning the second term, let $$A_l = \varrho(\nabla \hat
u_h|_{K_{l,1}})\ \partial_{\nu} \hat u_h|_{K_{l,1}} - \varrho(\nabla
\hat u_h|_{K_{l,2}})\ \partial_{\nu} \hat u_h|_{K_{l,2}},$$ where
again $l \subset \bar K_{l,1} \cap \bar K_{l,2}$, and the unit normal $\nu$
points outward of $K_{l,1}$. Therefore
\begin{eqnarray*}
- \langle G'\hat u_h, e - \pi e\rangle &=& - \sum_{K \in \mathcal{T}_h} \int_{\partial K} \varrho(\nabla \hat u_h)\  \partial_\nu \hat u_h|_{\partial K}\ (e-\pi e) \\
& = & - \sum_{l \not\subset \pom } \int_l A_l (e-\pi e) -  \sum_{l \subset \pom} \int_l  \varrho(\nabla \hat u_h)\ \partial_{\nu} \hat u_h|_l\ (e-\pi e).
\end{eqnarray*}
Repeating the analysis of \cite{cly}, Theorem 5.1, with the help of Lemma~\ref{gradlemma} gives $$- \sum_{l \not\subset \pom } \int_l A_l (e-\pi e) \lesssim \eta_{gr}^2 + \varepsilon (|e|_{(1, \hat u_h, p)}^2 + \eta_{gr}^2).$$
Thus
\begin{eqnarray*}
\|e, \tilde e, \epsilon\|_{Y^p}^p &\lesssim & |e|_{(1,\hat u, p)}^2 + \|e|_\pom + \tilde e\|_\Wg^2 + \|\epsilon\|_{W^{-\frac{1}{2},2}(\pom)}^2 \\
&\lesssim& \eta_f^2+\varepsilon (\eta_{gr}^2 + |e|_{(1,\hat u, p)}^2)+ \eta_{gr}^2
+ \varepsilon (|e|_{(1, \hat u_h, p)}^2 + \eta_{gr}^2)\\
& & 
 +\int_\gs \left\{-\varrho(\nabla \hat u_h)\ \partial_{\nu} \hat
u_h|_\gs (\tilde e_h - \tilde e)+g (|\tilde e_h +\hat v_h| - |\tilde
e +\hat v_h|)\right\} \\
& & - \int_\pom \varrho(\nabla \hat u_h)\
\partial_{\nu} \hat
u_h|_\pom\ ((e-\pi e)|_\pom +\tilde e - \tilde e_h)\\
&&+\langle t_0-\mathcal{W}(\hat u_h|_\pom+\hat v_h-u_0) -
(\mathcal{K}'-1) \hat \phi_h,
(e -\pi e)|_\pom +\tilde e - \tilde e_h\rangle \\
& & -\langle \nu-\nu_h, \mathcal{V} \hat \phi_h +
(1-\mathcal{K})(\hat u_h|_\pom+\hat v_h-u_0)\rangle.
\end{eqnarray*}
We bound the second, third + fourth as well as the final line
individually. Cauchy-Schwarz and Young's inequality allow to
estimate the last term by
$$\varepsilon \|e|_\pom + \tilde e\|_\Wg^2 + \varepsilon \|\epsilon\|_{W^{-\frac{1}{2},2}(\pom)}^2+\varepsilon^{-1}\ \|\mathcal{V} \hat
\phi_h + (1-\mathcal{K})(\hat u_h|_\pom+\hat v_h-u_0)\|_{\Wg}^2,$$
and the latter by $\eta_S^2$ (cf.~\cite{cast}). The third and fourth
lines are estimated by (cf.~\cite{cast})
$$\|-\varrho(\nabla \hat u_h)\ \partial_{\nu} \hat
u_h+t_0-\mathcal{W}(\hat u_h|_\pom+\hat v_h-u_0) - (\mathcal{K}'-1)
\hat \phi_h\|_{W^{-\frac{1}{2},2}(\pom)} \|(e-\pi e)|_\pom +\tilde e
\|_{\Wg}$$ which lead to $\eta_{\partial}$, where we have
choosen $\tilde e_h=0$, i.e. $v_h=\hat v_h$.
Finally,
using the triangle inequality,
the second line is simplified as follows:
\begin{align*}
&\int_\gs \{-\varrho(\nabla \hat u_h)\ \partial_{\nu} \hat u_h|_\gs
(\tilde e_h - \tilde e)+g (|\tilde e_h +\hat v_h| - |\tilde
e +\hat v_h|)\}\\
&\leq\int_\gs \{-\varrho(\nabla \hat u_h)\ \partial_{\nu} \hat u_h|_\gs
(\tilde e_h - \tilde e)+g|\tilde e_h -\tilde e |\}\\
&=\int_\gs \{\varrho(\nabla \hat u_h)\ \partial_{\nu} \hat u_h|_\gs
\tilde e +g|\tilde e |\}\\
&\leq \|\varrho(\nabla \hat u_h)\ \partial_{\nu} \hat
u_h|_\gs\|_{W^{-\frac{1}{2},2}(\gs)}\|\tilde
e\|_{W^{\frac{1}{2},2}(\gs)}
+\|g\|_{W^{-\frac{1}{2},2}(\gs)}\|\tilde e\|_{W^{\frac{1}{2},2}(\gs)}
.
\end{align*}
We may use the Cauchy-Schwartz inequality and the inverse inequality,
leading to $\eta_g$.

\end{proof}

\section{Numerical results}\label{sec:num}
With the subset $\Lambda_h$ of $\Whgs$ given by
\[
 \Lambda_h=\{\sigma_h\in \Whgs\,:\,|\sigma_h(x)|\leq 1\mbox{ a.e. on }
\Gamma_s\},
\]
we can define an Uzawa algorithm for solving the variational
inequality analogously to \cite{mast}. In order to introduce this
algorithm, let $P_\Lambda$ be the projection of $\Whgs$ onto
$\Lambda_h$, i.e. for every nodal point of the mesh
$\cT_h|_{\Gamma_s}$ holds
 $\delta\mapsto P_\Lambda(\delta)=\sup\{-1,\inf(1,\delta)\}$.
\begin{algorithm}[Uzawa]\quad\\
\begin{enumerate}
\item Choose $\sigma^0_h\in\Lambda_h$.
\item For $n=0,1,2,\ldots$
find $(u^n_h,v^n_h)\in X^p_h$ such that
\[\label{equ:uzca}
  \langle G' u^n_h, u_h\rangle + \langle S_h(u^n_h|_\pom+v^n_h), u_h|_\pom + v_h\rangle
+\int_{\Gamma_s} g\sigma^n_hv_h\,ds
  =\lambda_h(u_h,v_h)
\]
for all $(u_h,v_h)\in X^p_h$.
\item Set
\[\label{equ:uzs}
 \sigma^{n+1}_h=P_\Lambda(\sigma^n_h+\rho g v^n_h),
\]
where $\rho>0$ is a sufficiently small parameter that will be specified
later.
\item Repeat with 2. until a convergence criterion is satisfied.
\end{enumerate}
\end{algorithm}

In our first example the model problem is defined on the
L-shape with
$\Omega=[-\frac14,\frac14]^2\backslash[0,\frac14]^2$,
$\Omega^c=\R^2\backslash\Omega$.
The friction part of the
interface is $\Gamma_s=\overline{(-\frac14,-\frac14)(\frac14,-\frac14)}
           \cup\overline{(-\frac14,-\frac14)(-\frac14,\frac14)}$,
see Figure~\ref{fig:geo}.

In this example we choose $\varrho(t)=(\varepsilon+t)^{p-2}$, with
$p=3$ and $\varepsilon=0.00001$. Our volume and boundary data are
given by $f=0$ and $u_0=r^{2/3}\sin\frac23(\varphi-\frac\pi2)$,
$t_0=\partial_\nu u_0|_\pom$. The friction parameter is $g=0.5$,
leading to slip conditions on the interface. We have applied the
Uzawa algorithm as introduced above with the damping parameter
$\rho=25$ to solve the variational inequality. The nonlinear
variational problem in the Uzawa algorithm is then solved by
Newton's method in every Uzawa-iteration step.

In Table~\ref{tab:uni} we give the degrees of freedom, the value
$J_h(\hat u_h,\hat v_h)$ and the error measured with the help of
$J$, i.e. $\delta J=J_h(\hat u_h,\hat v_h)-J(\hat u,\hat v)$, where
we have obtained the value $J(\hat u,\hat v)$ by extrapolation of
$J_h(\hat u_h,\hat v_h)$. Due to the slip condition, we need only a
few Uzawa steps. But as a consequence of the degeneration of the
system matrix, due to the nonlinearity, the iteration numbers for
the MINRES solver, applied to the linearized system, are very high,
leading to large computation times. The convergence rate $\alpha_J$
is suboptimal, due to the presence of singularities, in the boundary
data as well, as due to the change of boundary conditions.

\begin{figure}[htb]
\begin{center}
\setlength{\unitlength}{2cm}
\begin{picture}(2.4,2.4)(-1.2,-1.2)
\put(-1,-1){\line(1,0){2}}
\put(-1,-1){\line(0,1){2}}
\put( 1,-1){\line(0,1){1}}
\put(-1, 1){\line(1,0){1}}
\put( 0, 0){\line(1,0){1}}
\put( 0, 0){\line(0,1){1}}
\put(-0.3,-0.3){\makebox{$\Omega$}}
\put(0.3,0.3){\makebox{$\Omega^c$}}
\put(1.1,-0.5){\makebox{$\Gamma_t$}}
\put(-1.15,0.0){\makebox{$\Gamma_s$}}
\put(-1.1,-1.1){\line(1,0){2.1}}
\put(-1.1,-1.1){\line(0,1){1.0}}
\put(-1.1, 0.2){\line(0,1){0.8}}
\put(-1.15,1.0){\line(1,0){0.1}}
\put(1.0,-1.15){\line(0,1){0.1}}
\end{picture}
\end{center}
\caption{\label{fig:geo}Geometry and interface of the model problem}
\end{figure}
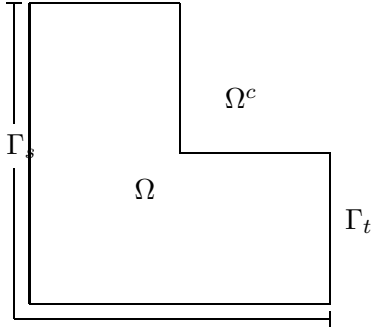

\begin{table}[htb]
\begin{center}
\begin{tabular}{rlllcl}
DOF & $J_h(\hat u_h,\hat v_h)$ & $\delta J$ & $\alpha_{J}$ & $It_{\rm Uzawa}$ & $\tau(s)$ \\ \hline
28&-0.511609&0.017249&---&2&0.190 \\
80&-0.517938&0.010920&-0.435&2&0.640 \\
256&-0.521857&0.007001&-0.382&2&2.440 \\
896&-0.524293&0.004566&-0.341&2&11.05 \\
3328&-0.525841&0.003017&-0.316&2&61.85 \\
12800&-0.526865&0.001993&-0.308&2&437.5 \\
50176&-0.527571&0.001287&-0.320&2&4218. \\
\end{tabular}
\end{center}
\caption{\label{tab:uni} Convergence rates and Uzawa steps for uniform
meshes (Example 1)}
\end{table}

In our second example we use the same model geometry as before (see
Fig.~\ref{fig:geo}). Here we choose the friction boundary
$\Gamma_s=\emptyset$.
Therefore
our model problem reduces to a non-linear
p-Laplacian FEM-BEM coupling problem, where we can prescribe the solution.

In this example we choose $\varrho(t)=(\varepsilon+t)^{p-2}$, with
$p=3$ and $\varepsilon=0.00001$. We prescribe the solution by
$u_1=r^{2/3}\sin\frac23(\varphi-\frac\pi2)$ and $u_2=0$.
Then the boundary data $u_0,t_0$ and volume data $f$ are given by
$u_0=u_1|_\Gamma$, $t_0=\varrho(|\nabla u_1|)\partial_\nu u_1$ and
$f=-\Div(\varrho(|\nabla u_1|)\nabla u_1)$.

In the following we give errors in the $\|\cdot\|_{W^{1,p}(\Omega)}$
norm and in the quasinorm $|u-u_h|_Q=\|u-u_h\|_{(1,u_h,p)}$.

In Tab.~\ref{tab:h4} we give the errors, convergence rates, number
of Newton iterations $It_{Newton}$ and the computing time for the
uniform h-version with rectangles. We observe that the convergence
rate in the quasi-norm $|\cdot|_{Q}$ is better than in the
$\|\cdot\|_{W^{1,3}(\Omega)}$-norm. The number of Newton iterations
appears to be bounded.

In Tab.~\ref{tab:h3} for the uniform h-version with triangles, we
give the errors, convergence rates, error estimator $\eta$,
efficiency indices $\delta_u/\eta$ for the
$\|\cdot\|_{W^{1,3}(\Omega)}$-norm and $\delta_q/\eta$ for the
$|\cdot|_{Q}$-norm, number of Newton iterations and the computing
time. Again, here we observe that the convergence rate in the
quasi-norm $|\cdot|_{Q}$ is better than in the
$\|\cdot\|_{W^{1,3}(\Omega)}$-norm and the number of Newton
iterations is bounded. The efficiency index $\delta_u/\eta$ appears
to be constant, whereas the efficiency index $\delta_q/\eta$ appears
to be decreasing.

Tab.~\ref{tab:adap} gives the corresponding numbers for the adaptive
version, using a blue-green refining strategy for triangles and
refining the 10\% elements with the largest indicators. Here we
observe that the convergence rates for both norms are very similar
and that both efficiency indices are bounded.

Figure~\ref{fig:L2} give the errors for all methods in the
$\|\cdot\|_{W^{1,3}(\Omega)}$-norm and the $|\cdot|_{Q}$ quasi-norm
together with the error indicators for the uniform and adaptive
methods.

Figure~\ref{fig:meshes} presents the sequence of meshes generated by
the adaptive refinement strategy. We clearly observe the refinement
towards the reentrant corner with the singularity of the solution.

\begin{table}
\begin{center}
\begin{tabular}{rl@{}rlrcr}
DOF & $\|u-u_h\|_{1,3}$ & $\alpha$
    & $|u-u_h|_{Q}$ & $\alpha$
& $It_{Newton}$ & $\tau(s)$ \\ \hline
21&0.1711499&---&0.1293512&---&22&0.224 \\
65&0.1308635&-0.238&0.0860870&-0.360&22&0.424 \\
225&0.1039326&-0.186&0.0612225&-0.274&23&1.668 \\
833&0.0826578&-0.175&0.0438478&-0.255&23&6.804 \\
3201&0.0657091&-0.170&0.0314280&-0.247&23&27.28 \\
12545&0.0522196&-0.168&0.0225589&-0.243&24&120.8 \\
49665&0.0414910&-0.167&0.0162319&-0.239&24&560.1 \\
197633&0.0329617&-0.167&0.0117169&-0.236&24&2678.
\end{tabular}
\end{center}
\caption{\label{tab:h4}Errors, convergence rates (Example 2, uniform
  mesh with rectangles)}
\end{table}
\begin{table}
\begin{center}
\begin{tabular}{rl@{}rl@{\quad }rrcc@{}c@{}r}
DOF & $\|u-u_h\|_{1,3}$ & $\alpha$
    & $|u-u_h|_{Q}$ & $\alpha$
    & $\eta$ & $\delta_u/\eta$ & $\delta_q/\eta$ &
$It_{New}$ & $\tau(s)$ \\ \hline
21&0.1945908&---&0.1510064&---&1.027&0.190&0.147&22&0.620 \\
65&0.1535874&-0.209&0.1081632&-0.295&0.690&0.223&0.157&22&2.212 \\
225&0.1219287&-0.186&0.0774765&-0.269&0.516&0.236&0.150&22&8.617 \\
833&0.0969249&-0.175&0.0555005&-0.255&0.394&0.246&0.141&23&36.00 \\
3201&0.0770270&-0.171&0.0396882&-0.249&0.304&0.253&0.131&23&144.2 \\
12545&0.0611994&-0.168&0.0283778&-0.246&0.236&0.260&0.120&24&608.7 \\
49665&0.0486160&-0.167&0.0203130&-0.243&0.184&0.265&0.111&24&2530. \\
197633&0.0386151&-0.167&0.0145686&-0.241&0.144&0.269&0.102&24&11000
\end{tabular}
\end{center}
\caption{\label{tab:h3}Errors, onvergence rates, estimator
  $\eta$, reliability $\delta_u/\eta$ and $\delta_q/\eta$ (Example 2,
  uniform mesh with triangles)}
\end{table}

\begin{table}
\begin{center}
\begin{tabular}{rl@{}rl@{\quad }rrcc@{}c@{}r}
DOF & $\|u-u_h\|_{1,3}$ & $\alpha$
    & $|u-u_h|_{Q}$ & $\alpha$
    & $\eta$ & $\delta_u/\eta$ & $\delta_q/\eta$ &
$It_{New}$ & $\tau(s)$ \\ \hline
21&0.1945908&---&0.1510064&---&1.027&0.190&0.147&22&0.196 \\
32&0.1602214&-0.461&0.1205155&-0.535&0.804&0.199&0.150&22&0.332 \\
54&0.1275298&-0.436&0.0918131&-0.520&0.603&0.212&0.152&22&0.648 \\
93&0.1019990&-0.411&0.0699054&-0.501&0.442&0.231&0.158&22&1.132 \\
152&0.0821754&-0.440&0.0540462&-0.524&0.325&0.253&0.166&23&2.000 \\
249&0.0679251&-0.386&0.0449420&-0.374&0.246&0.276&0.183&23&3.352 \\
400&0.0558447&-0.413&0.0369614&-0.412&0.190&0.294&0.194&23&5.700 \\
625&0.0439784&-0.535&0.0277857&-0.639&0.148&0.297&0.188&24&9.896 \\
986&0.0352491&-0.485&0.0217361&-0.539&0.116&0.305&0.188&24&17.45 \\
1528&0.0279287&-0.531&0.0167409&-0.596&0.091&0.308&0.184&25&31.16 \\
2322&0.0222760&-0.540&0.0129489&-0.614&0.071&0.312&0.181&25&53.98 \\
3620&0.0177640&-0.510&0.0102552&-0.525&0.056&0.316&0.182&25&106.7 \\
5544&0.0142059&-0.524&0.0080233&-0.576&0.044&0.320&0.181&25&205.3 \\
8449&0.0112965&-0.544&0.0063426&-0.558&0.035&0.322&0.181&26&422.4 \\
12810&0.0090396&-0.536&0.0050706&-0.538&0.028&0.325&0.183&26&1060. \\
19222&0.0072288&-0.551&0.0040370&-0.562&0.022&0.329&0.184&26&2400. \\
29006&0.0057984&-0.536&0.0032478&-0.529&0.018&0.333&0.186&27&5460. \\
43593&0.0046615&-0.536&0.0026230&-0.524&0.014&0.337&0.190&27&13000
\end{tabular}
\end{center}
\caption{\label{tab:adap}p-Laplacian (adaptive), convergence rates,
estimator
  $\eta$, reliability $\delta_u/\eta$ and $\delta_q/\eta$ }
\end{table}

\begin{figure}[htb]
\centerline{\resizebox{7cm}{!}{\includegraphics*{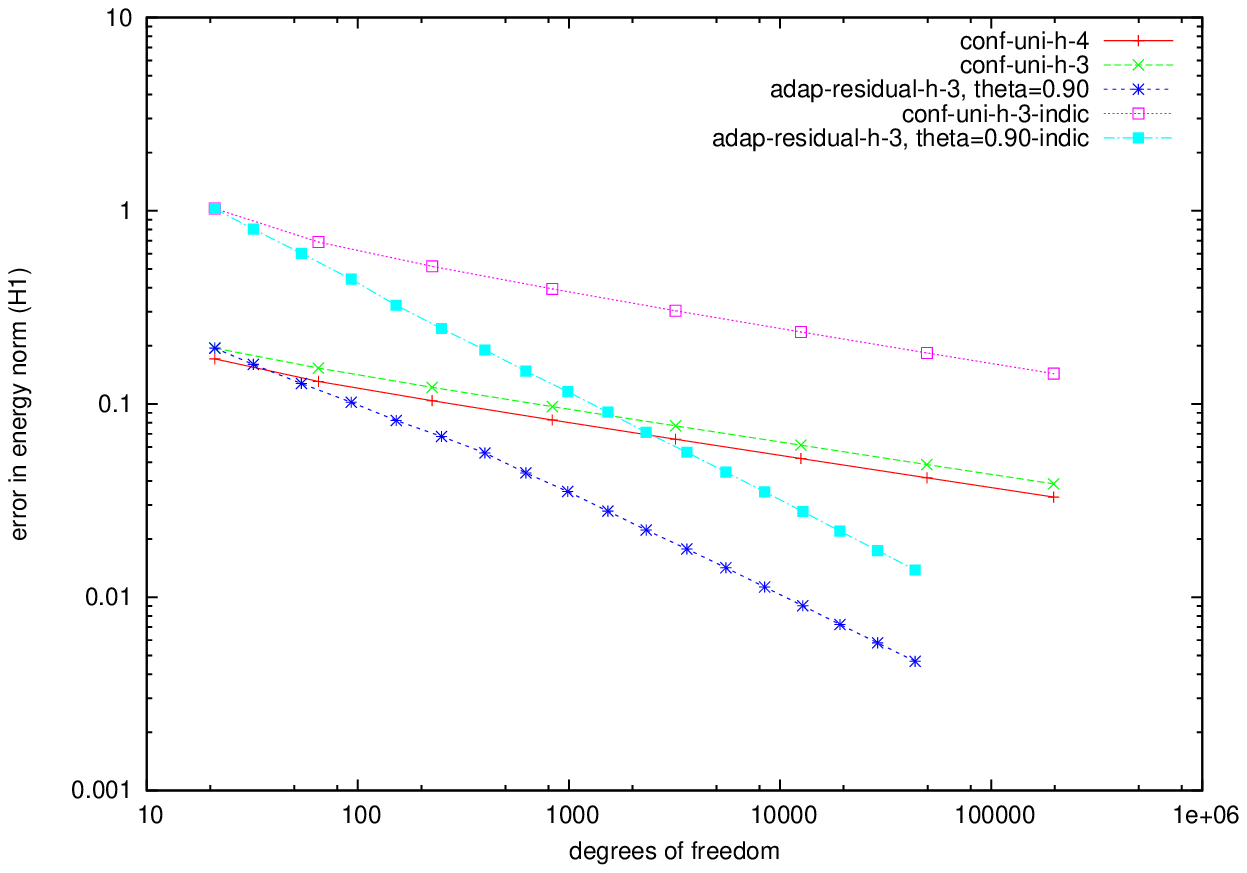}}\hfill
            \resizebox{7cm}{!}{\includegraphics*{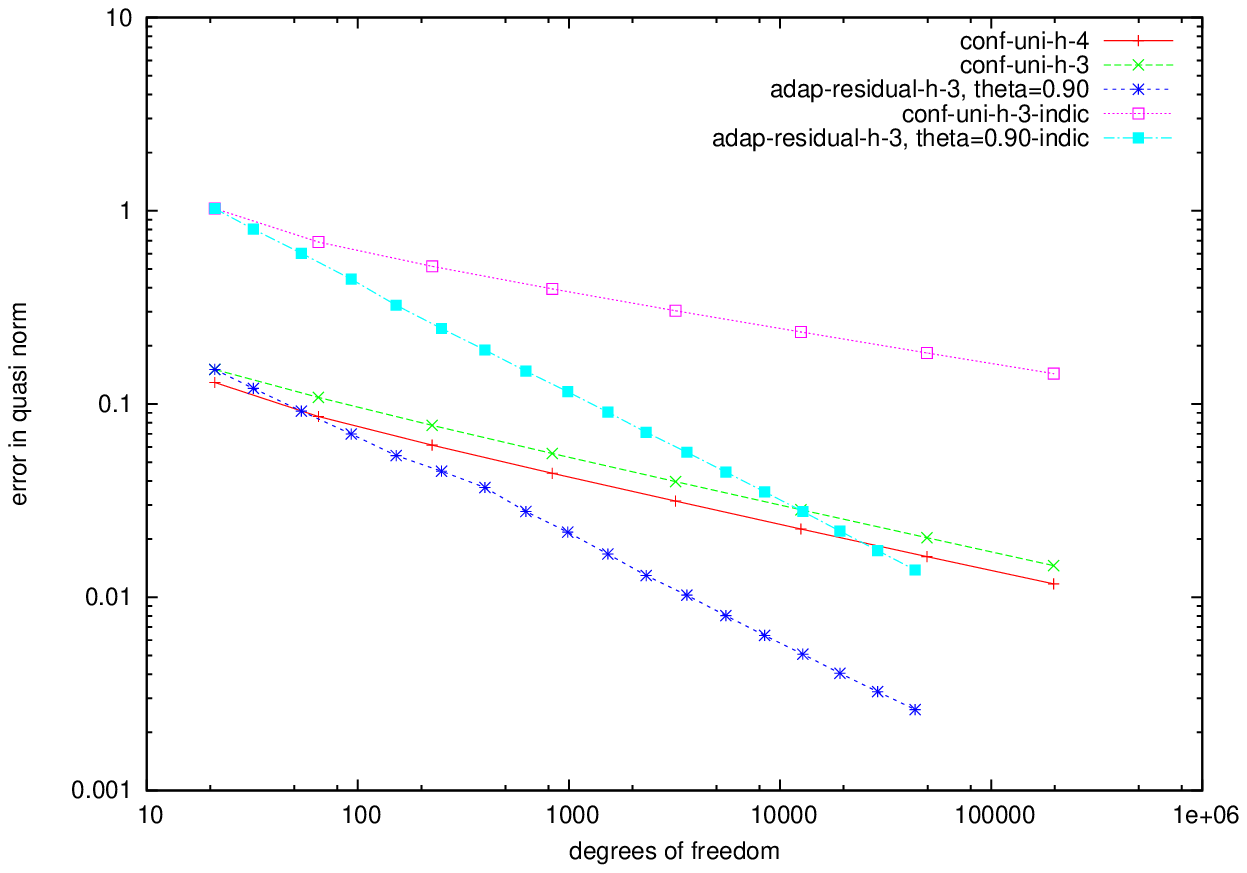}}}
\caption{\label{fig:L2} $\|u-u_n\|_{W^{1,3}(\Omega)}$ (left)
  and $|u-u_n|_{Q}$ (right).}
\end{figure}

\begin{figure}
\resizebox{3.3cm}{!}{\includegraphics*{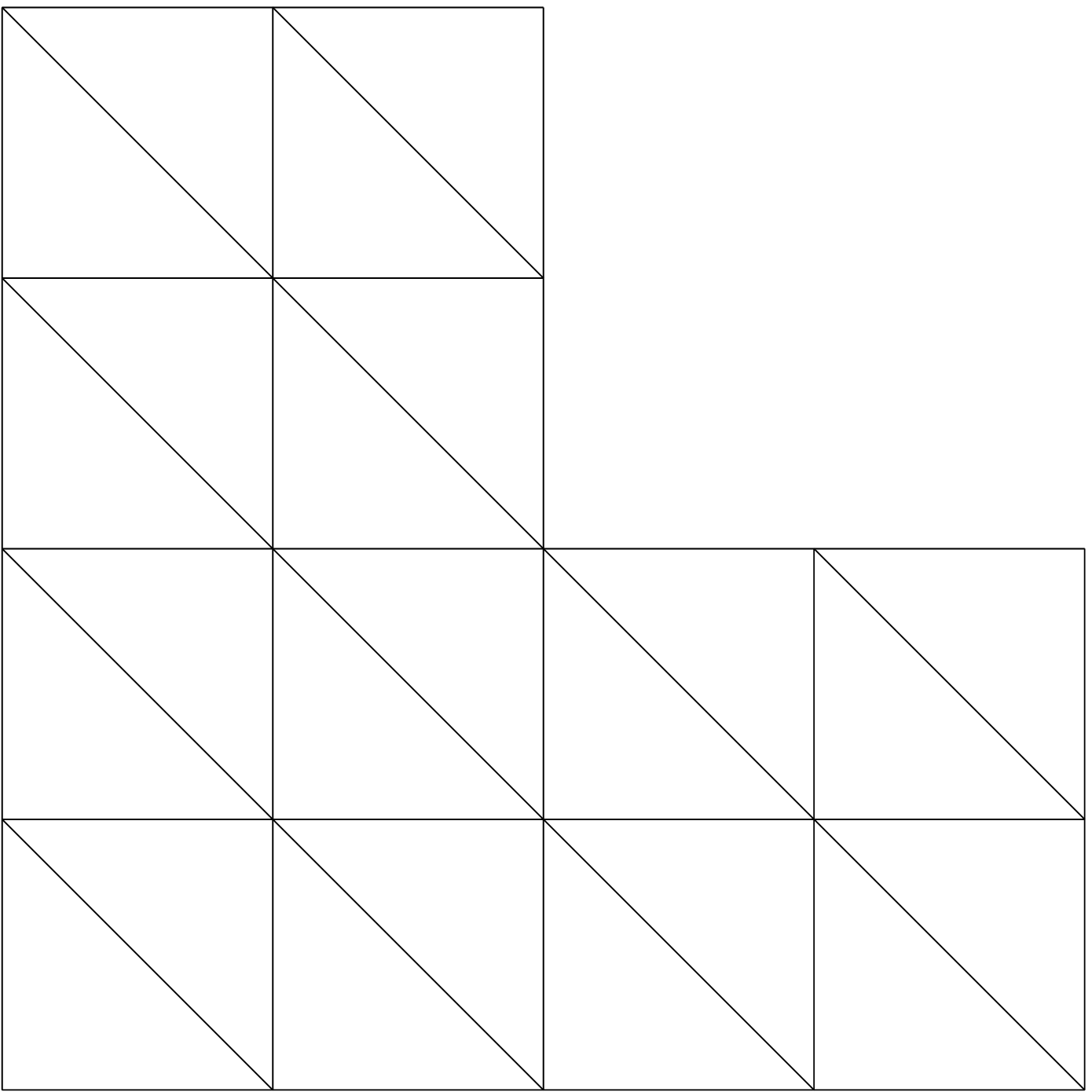}}\qquad
\resizebox{3.3cm}{!}{\includegraphics*{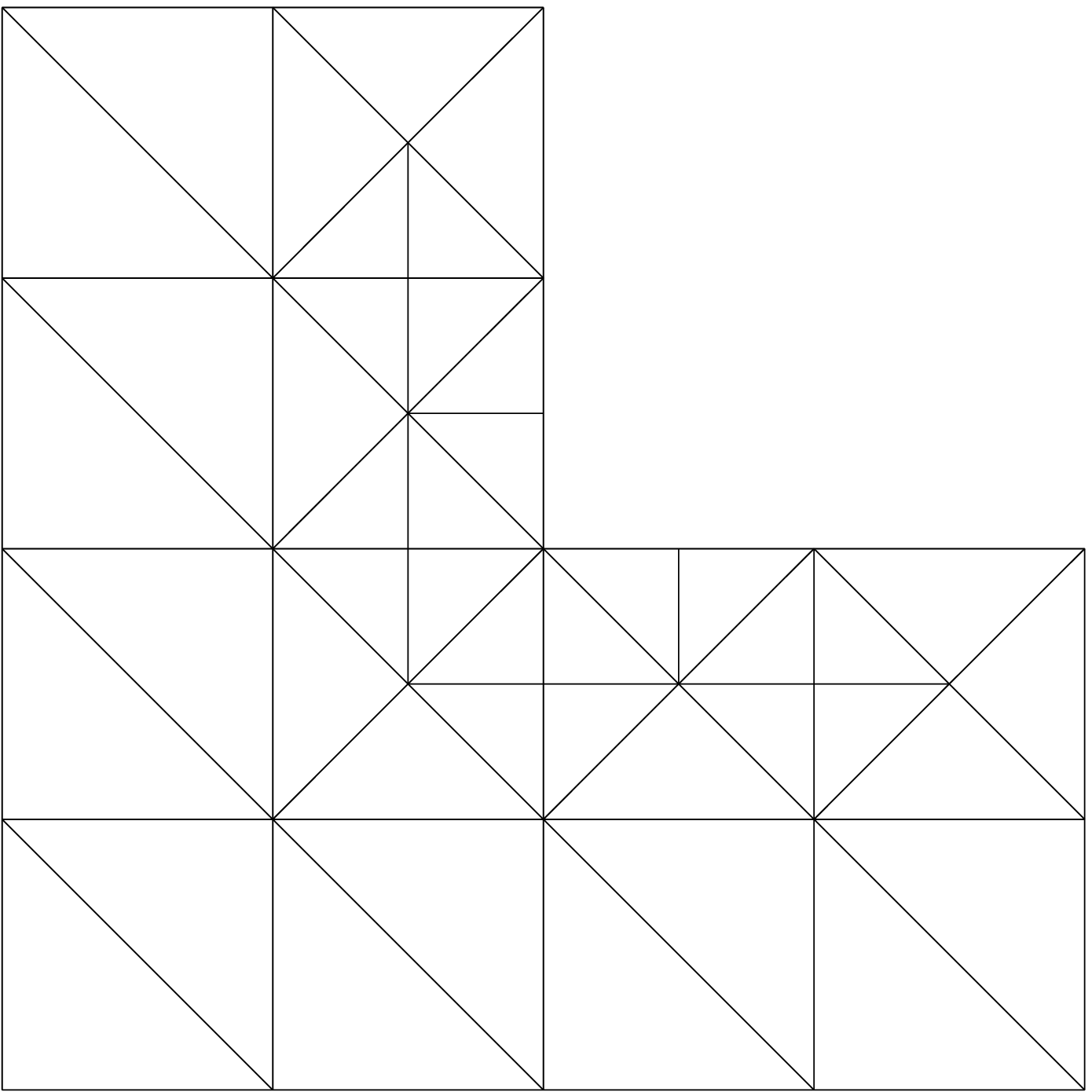}}\qquad
\resizebox{3.3cm}{!}{\includegraphics*{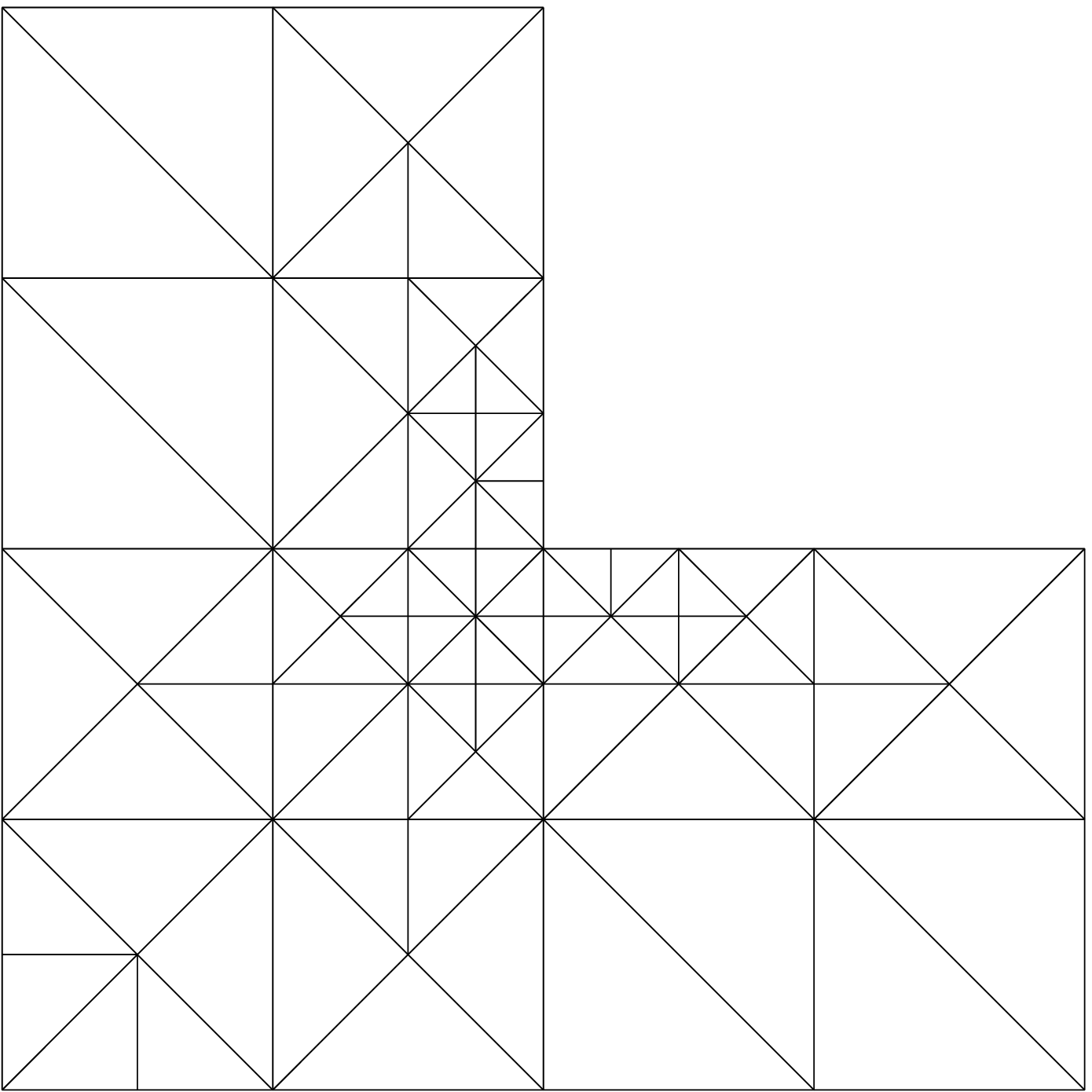}}\\

\resizebox{3.3cm}{!}{\includegraphics*{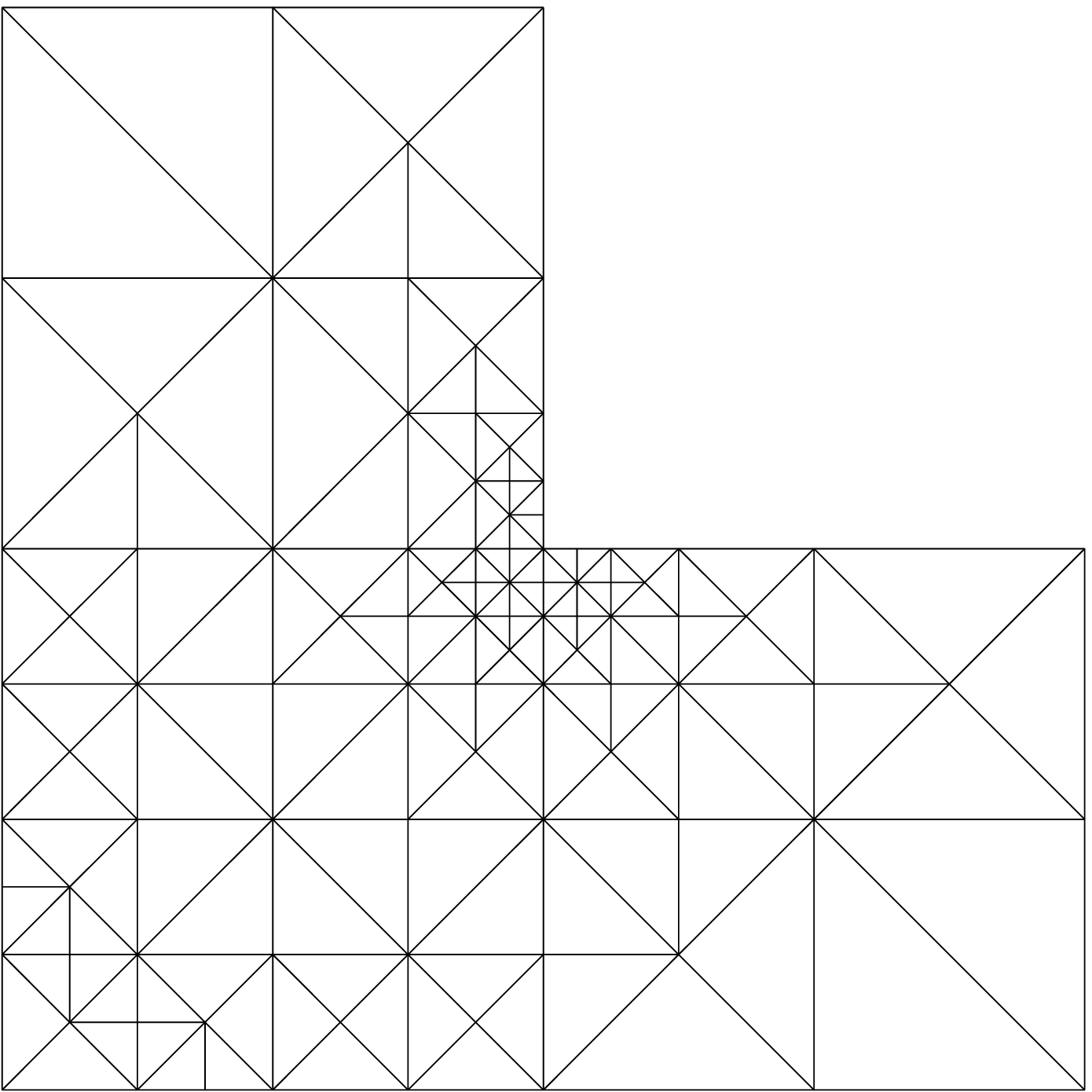}}\qquad
\resizebox{3.3cm}{!}{\includegraphics*{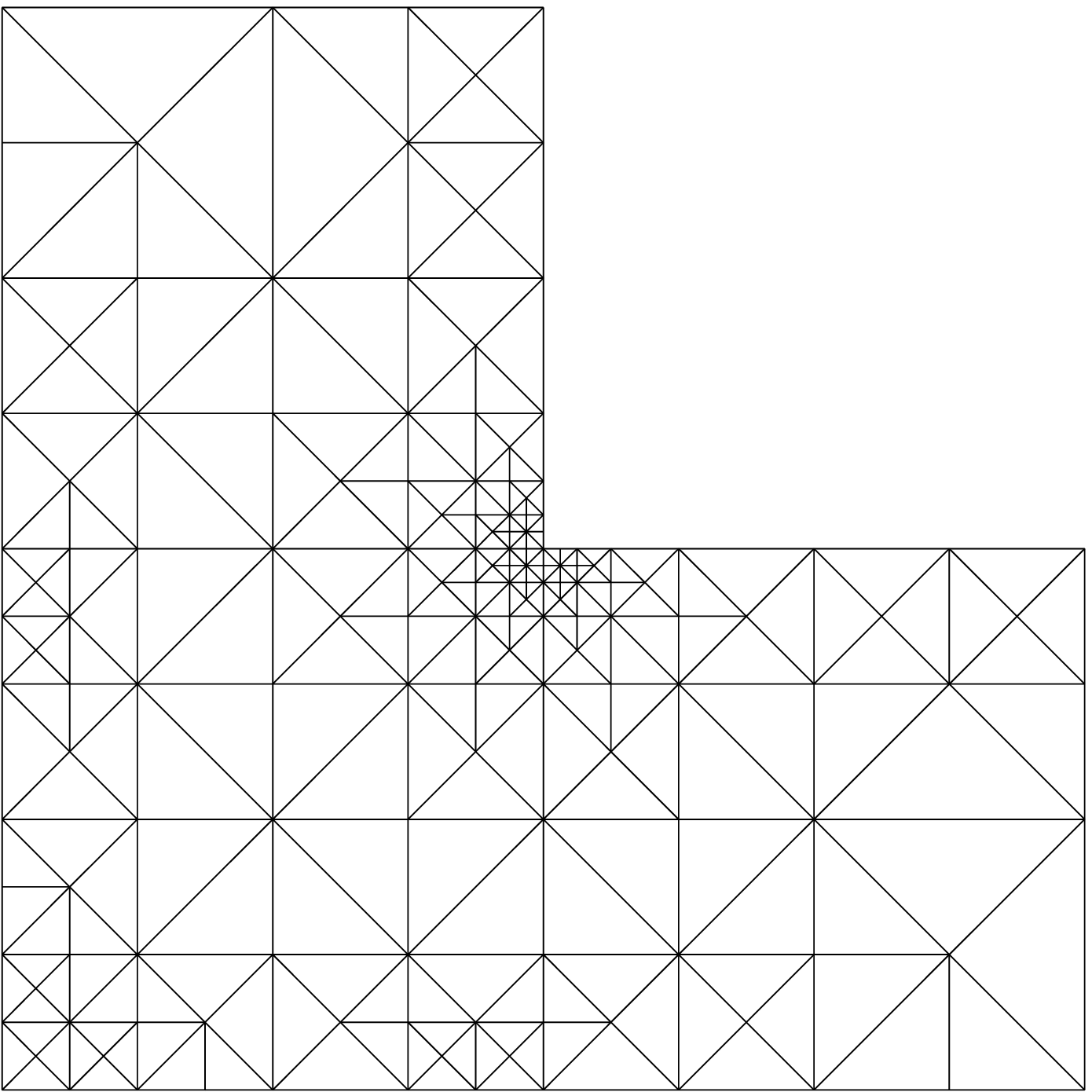}}\qquad
\resizebox{3.3cm}{!}{\includegraphics*{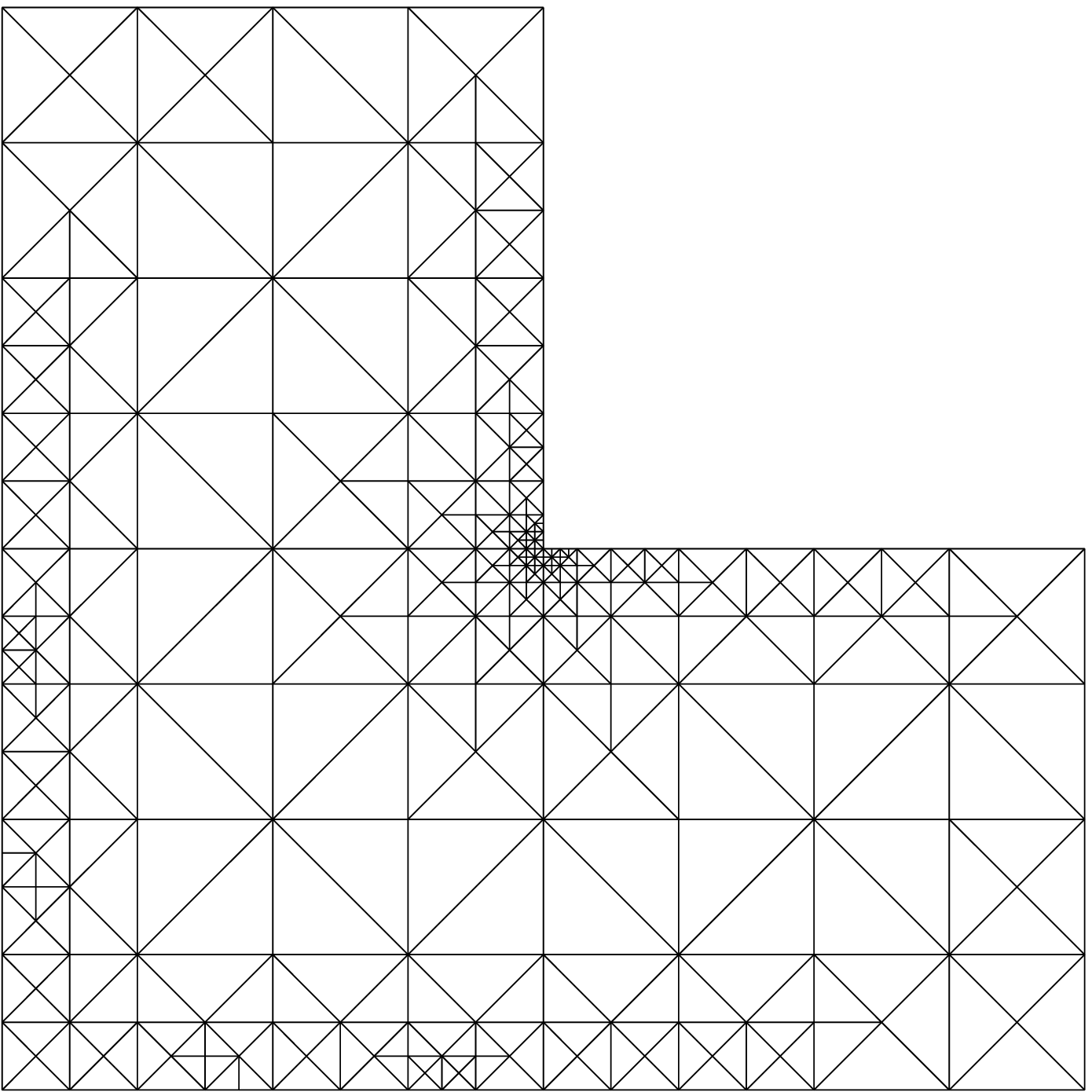}}\\

\resizebox{3.3cm}{!}{\includegraphics*{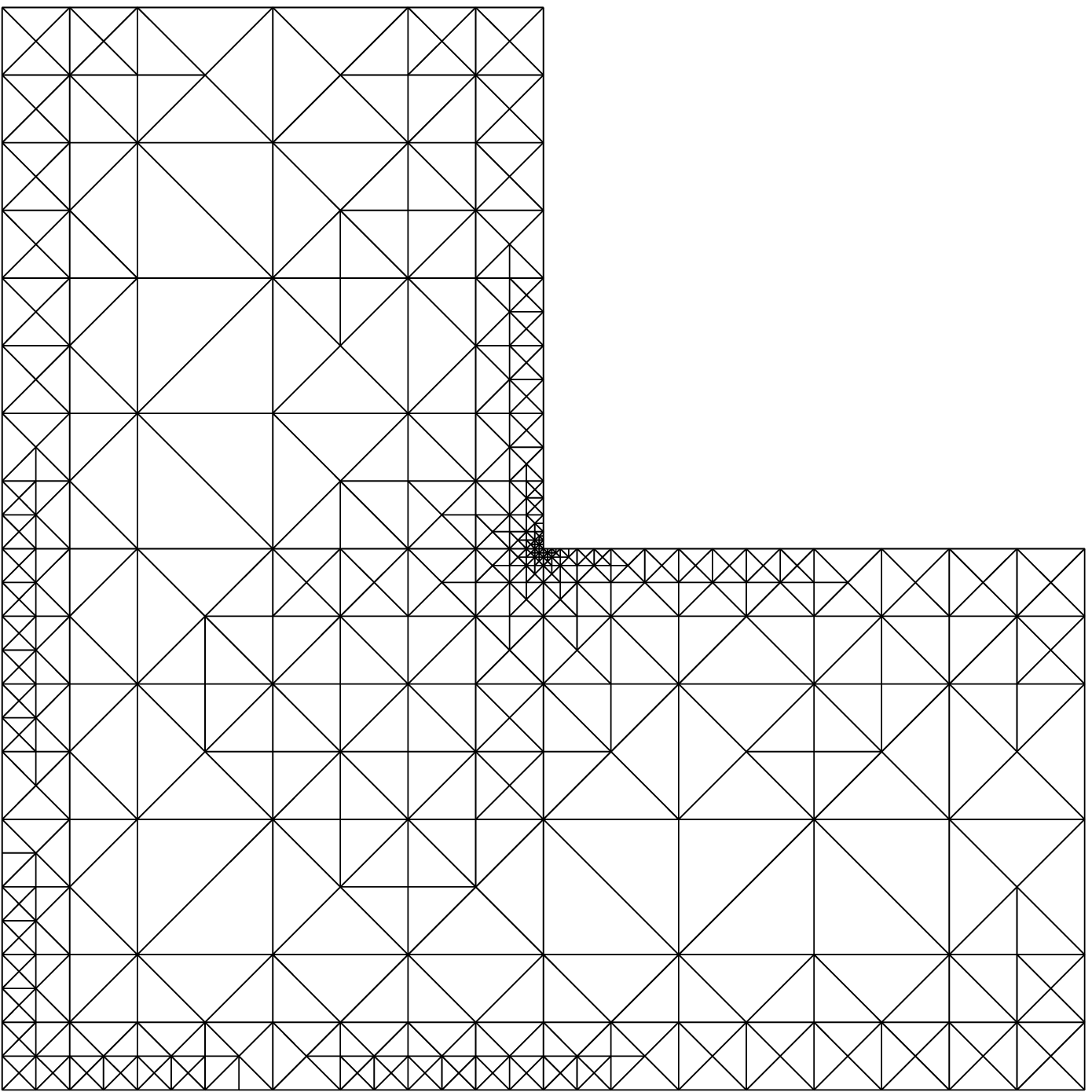}}\qquad
\resizebox{3.3cm}{!}{\includegraphics*{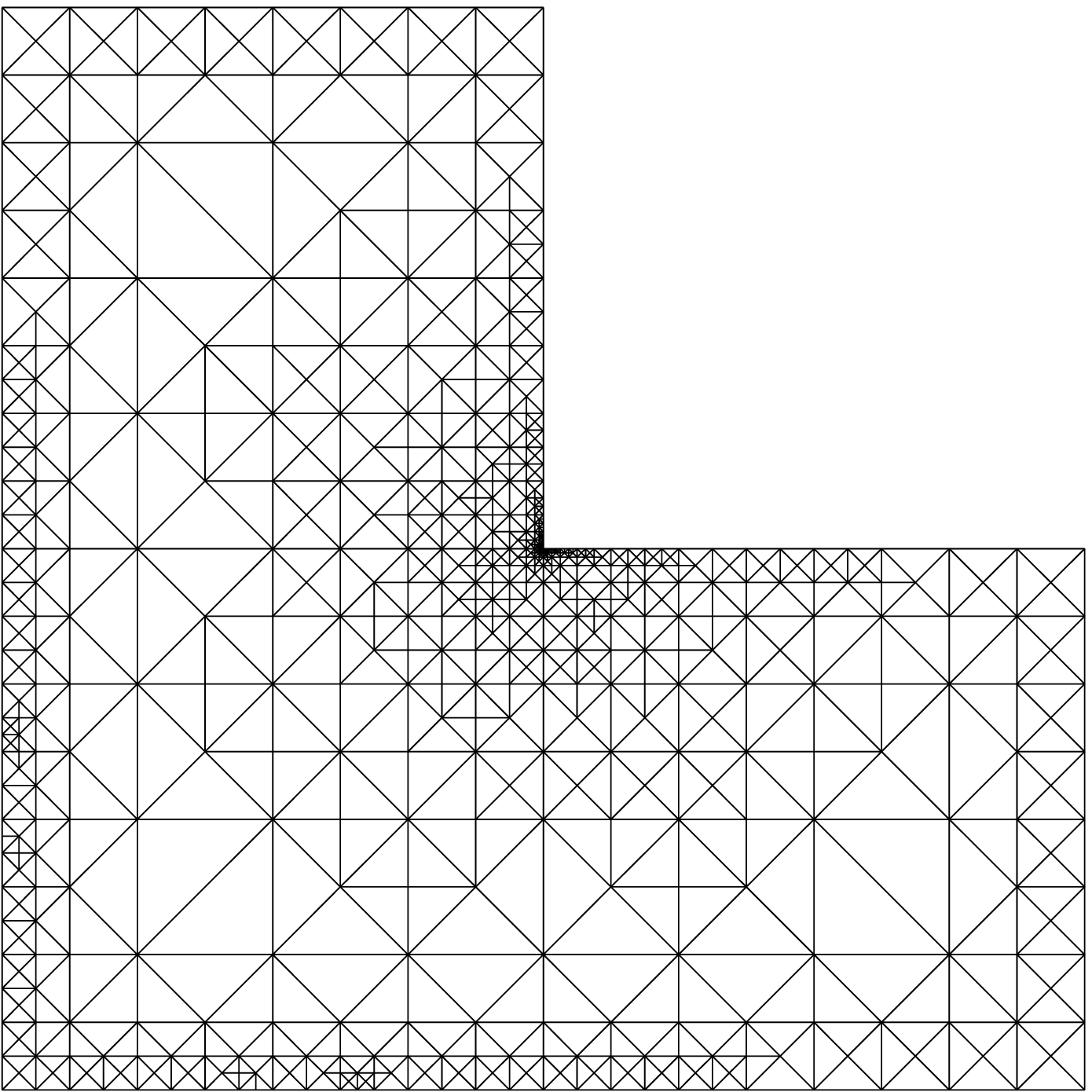}}\qquad
\resizebox{3.3cm}{!}{\includegraphics*{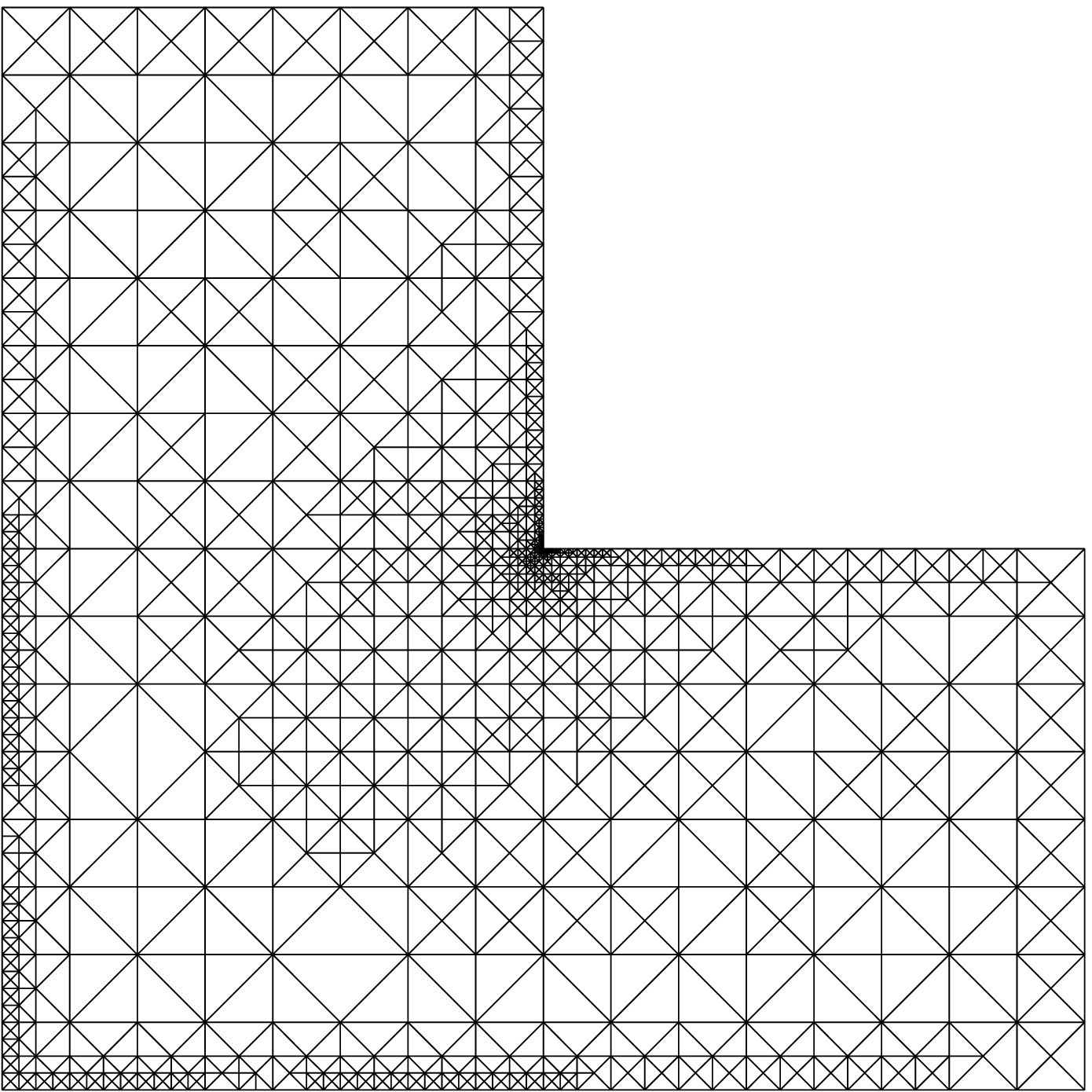}}\\

\resizebox{3.3cm}{!}{\includegraphics*{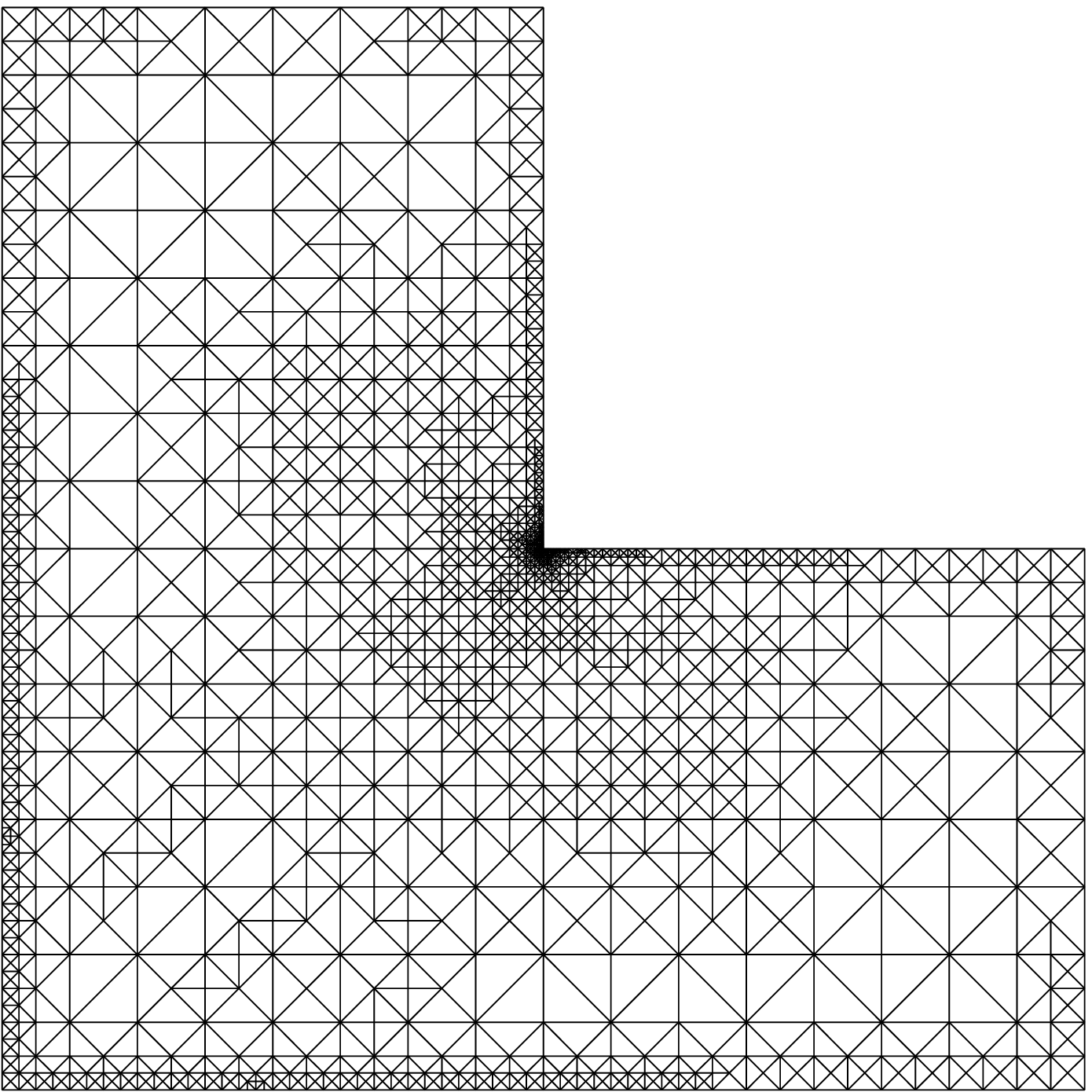}}\qquad
\resizebox{3.3cm}{!}{\includegraphics*{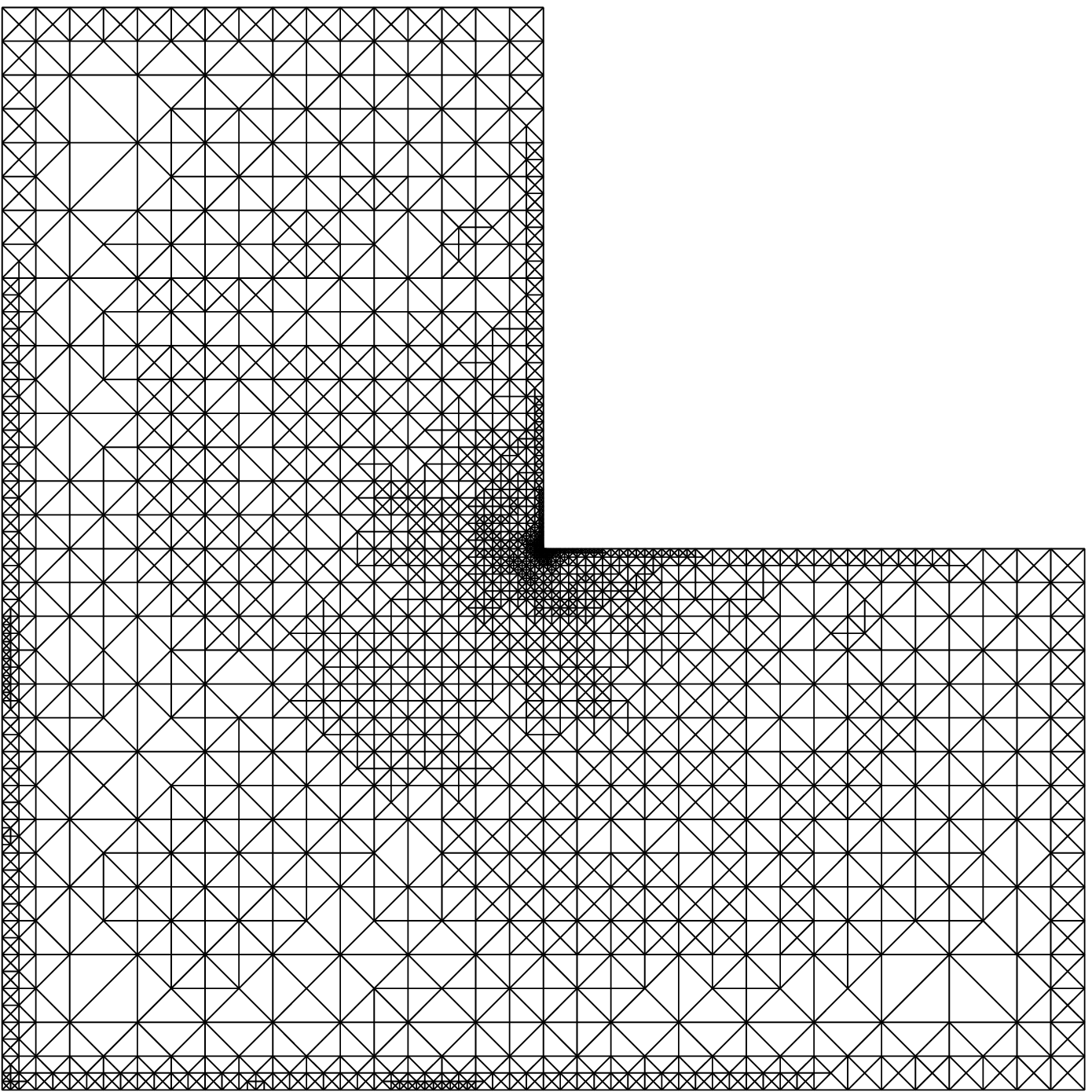}}\qquad
\resizebox{3.3cm}{!}{\includegraphics*{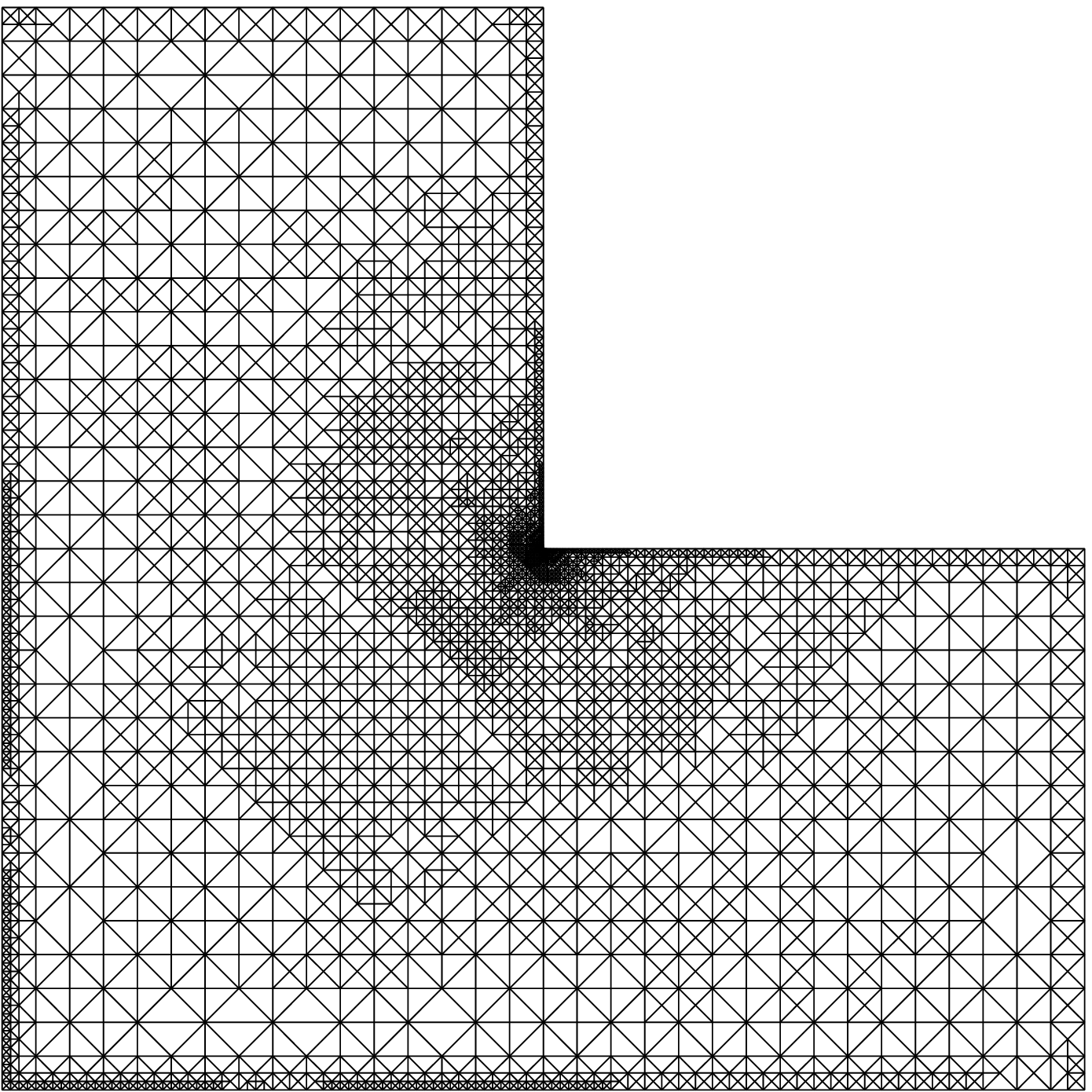}}



\caption{\label{fig:meshes} The first 12 meshes generated by the
  adaptive refinement algorithm}
\end{figure}

\end{document}